\documentclass[11pt]{amsart}
\usepackage{qtree}
\usepackage{tikz-cd}
\usepackage{mathtools}
\usepackage{latexsym}
\usepackage{amsfonts}
\usepackage{amssymb}
\usepackage{amsmath}
\usepackage{mathrsfs}
\usepackage{hyperref}

\newtheorem{definition}{Definition}[section] 
\newtheorem{theorem}[definition]{Theorem}
\newtheorem{lemma}[definition]{Lemma}
\newtheorem{proposition}[definition]{Proposition}
\newtheorem{corollary}[definition]{Corollary} 

\newtheorem{remark}[definition]{Remark}

\title[K\"ahler structure on certain $C^*$-dynamical systems]{K\"ahler structure on certain $C^*$-dynamical systems and\\
the noncommutative even dimensional tori}
\author{Satyajit Guin}
\address{Indian Institute of Science Education and Research, Mohali, Punjab 140306}
\email{satyamath@gmail.com , satyajit@iisermohali.ac.in}

\keywords{complex structure, K\"ahler structure, $C^*$-dynamical system, spectral triple, noncommutative torus}
\date{\today}
\subjclass[2010]{58B34, 46L87}

\begin{document}
\rmfamily

\begin{abstract}
Let $G$ be an even dimensional, connected, abelian Lie group and $(\mathcal{A}^\infty,G,\alpha,\tau)$ be a $C^*$-dynamical system equipped with a
faithful $G$-invariant trace $\tau$. We show that whenever it determines a $\varTheta$-summable even spectral triple, $\mathcal{A}^\infty$ inherits
a K\"ahler structure. Moreover, there are at least $\prod_{j=1,\,j\,odd}^{\,dim(G)}(dim(G)-j)$ different K\"ahler structures. In particular, whenever
$\mathbb{T}^{2k}$ acts ergodically on the algebra, it inherits a K\"ahler strcture. This gives a class of examples of noncommutative K\"ahler manifolds.
As a corollary, we obtain that all the noncommutative even dimensional tori, like their classical counterpart the complex tori, are noncommutative
K\"ahler manifolds. We explicitly compute the space of complex differential forms for the noncommutative even dimensional tori and show that the
category of holomorphic vector bundle over it is an abelian category. We also explain how the earlier set-up of Polishchuk-Schwarz for the holomorphic
structure on noncommutative two-torus follows as a special case of our general framework.
\end{abstract}

\maketitle

\section {Introduction}
Classical differential geometry was extended to the noncommutative world of $\,C^*$-algebras in the early $80$'s by Connes in (\cite{Con1}), and
subsequently in (\cite{Con2}). Many highly singular (and classically intractable) objects such as the dual of a discrete group, Penrose tilings or
quantum groups may be analyzed by applying cyclic cohomology, K-theory and other tools of noncommutative geometry. Apart from its own mathematical
beauty, several fruitful applications of noncommutative geometry in physics (see for e.g. \cite{SW},\cite{CMA},\cite{Con3}) have also been observed.
Despite much progress in noncommutative geometry in past $30$ years, noncommutative complex geometry is not developed that much yet. Connes-Cuntz
first outlined a possible approach to the idea of a complex structure in noncommutative geometry based on the notion of positive Hochschild cocycle
on an involutive algebra (\cite{CCu}). In (\cite{Con2}, Section VI.2) Connes shows explicitly that positive Hochschild cocycles on the algebra of
smooth functions on a compact oriented $2$-dimensional manifold encode the information needed to define a holomorphic structure on the surface.
However, the corresponding problem of characterizing holomorphic structures on $n$-dimensional manifolds via positive Hochschild cocycles is still open.

Coming to concrete examples, a detail study of complex structure on noncommutative two-torus and holomorphic vector bundles on them is carried out
in (\cite{PS},\cite{P}), taking motivation from (\cite{S},\cite{DS}). Complex structure on the Podle\'s sphere is studied in (\cite{Maj}) using a
frame bundle approach, and simultaneously but independently in (\cite{HK},\cite{HK1}) using a classification of the covariant first order differential
calculi of the irreducible quantum flag manifolds. Later in (\cite{KLS}), properties of the $q$-Dolbeault complex of (\cite{Maj}) are formalized and
it was shown to resemble in many aspects the analogous structure on the classical Riemann sphere. See (\cite{KM}) for the case of higher dimensional
quantum projective spaces. A more comprehensive version of noncommutative complex structure appeared later in (\cite{BS}) and complex structure on
quantum homogeneous spaces is studied in (\cite{Bu}). The main tool used in all these examples is the Woronowicz's differential calculus for quantum
groups (\cite{Wor}). In this algebraic setting, recently, the notion of K\"ahler structure has been introduced in (\cite{Bua}) for quantum homogeneous
spaces, taking the quantum flag manifolds as motivating family of examples. However, our approach (based on \cite{FGR2}) in this article is different
from this, taking the noncommutative torus as motivating example. We discuss it now.

In noncommutative geometry, a (noncommutative) manifold is described by a triple called spectral triple. That the notion of spectral triple is the
correct noncommutative generalization of classical manifolds is shown by Connes (\cite{Con6}). However, it
turns out that the notion of spectral triple is not quite appropriate to describe the higher geometric structures, e.g. symplectic, complex, Hermitian,
K\"ahler or hyper-K\"ahler structures, even in the classical setting. Around '$98$, a decent approach to noncommutative symplectic, complex, Hermitian,
K\"ahler and hyper-K\"ahler geometry is made by Fr\"{o}hlich et al. (\cite{FGR1},\cite{FGR2}) in the context of supersymmetric quantum theory. Unlike
the case of above discussed examples, where the approach is purely algebraic, methods of Fr\"{o}hlich et al. is geometric and analytic in the sense
that spectral triple lies at the heart of it and integration theory is built-in using $\beta$-KMS state. Taking inspiration from Witten's supersymmetric
approach to the Morse inequalities (\cite{Wit1}) and the work of Jaffe et al. on connections between cyclic cohomology and supersymmetry (\cite{Jaf}),
Fr\"{o}hlich et al. obtained the supersymmetric algebraic formulation of Riemannian, spin, symplectic, complex, Hermitian, K\"{a}hler and hyper-K\"{a}hler
geometry in (\cite{FGR1}, see also $\S 3$.B in \cite{Huy} for discussion), which then readily generalizes to the noncommutative geometry framework of
spectral triples in (\cite{FGR2}).  It is important to mention here that there are well known links between supersymmetric $\,\sigma$-models and the
geometry of manifolds (\cite{AF}). The approach of Fr\"{o}hlich et al. starts with a spectral triple and detects the precise analytic conditions
required to obtain complex, Hermitian, K\"ahler or hyper-K\"ahler structure. They have denoted these various higher geometric structures by
$\,N=1,\,N=2$ and $\,N=(n,n)$ with $n=1,2,4$, along the line of supersymmetry. The relationship among these geometric structures are vaguely denoted
by $N=(4,4)\preccurlyeq N=(2,2)\preccurlyeq N=(1,1)\preccurlyeq N=1$ to mean that the former is obtained from the latter by imposing certain additional
conditions. Among these, our concern in this article are the $N=1,\,N=(1,1)$ and $N=(2,2)$ K\"{a}hler geometries. Note that the $N=1$ data is
specified by a $\varTheta$-summable even spectral triple in noncommutative geometry, and the $N=(2,2)$ data extends the notion
of Hermitian and K\"ahler manifolds to noncommutative geometry. For precise definitions see Section $(2)$ Defn. (\ref{N=1 spectral data}\,--\,\ref{N=(2,2) spectral data}).
We will call these various higher geometric structures as the $N=\bullet$ or $N=(\bullet\,,\bullet)$ spectral data in this article. In the classical
case of a spin manifold $\,\mathbb{M}$, from the $N=(1,1)$ spectral data one may recover the graded algebra of differential forms on $\mathbb{M}$
and in particular, the exterior differential.

Motivated by the case of noncommutative two-torus in (\cite{FGR2}), we prove that any $N=1$ spectral data obtained from a $C^*$-dynamical system
equipped with a faithful invariant trace (\cite{Con1},\cite{CR},\cite{Rfl3}), where the Lie group acting is even dimensional, connected and abelian,
always extends to $N=(2,2)$ K\"ahler spectral data over the same algebra, i,e. it inherits a K\"ahler structure. Moreover, there are at least
$\prod_{j=1,\,j\,odd}^{\,dim(G)}(dim(G)-j)$ different K\"ahler structures. In particular, whenever $\,\mathbb{T}^{2k}$ acts ergodically on the
algebra, it inherits a K\"ahler strcture. This produces a class of examples of noncommutative K\"ahler manifolds. As a corollary, we obtain that all
the noncommutative even dimensional tori, like their classical counterpart the complex tori, are noncommutative K\"ahler manifolds. Note that in the
noncommutative situation, noncommutative two-torus was the only known example of noncommutative K\"ahler manifold (apart from the ones recently
produced in \cite{Bua}). We also study the category of holomorphic vector bundle and show that any free module over a Hermitian spectral data is a
holomorphic vector bundle. Then we consider the particular case of noncommutative even dimensional tori and explicitly compute the associated space
of complex differential forms. This helps us to show that the category of holomorphic vector bundle over noncommutative even dimensional tori is an
abelian category, extending earlier result in (\cite{PS},\cite{P}) for the case of noncommutative two-torus. We then explain how the earlier set-up
of Polishchuk-Schwarz (\cite{PS}) for the case of noncommutative two-torus follows as a special case of our general framework for $C^*$-dynamical
systems. For a $4k$-dimensional abelian Lie group whether the K\"ahler structure obtained here extends further to a hyper-K\"ahler structure is left
as an open question. We want to mention here one crucial point that our proof relies on the fact that the associated Lie algebra $\mathfrak{g}$ of
the even dimensional Lie group $G$ acting on $\mathcal{A}$ is abelian, i,e. $\mathfrak{g}=\mathbb{R}^{2k}$. This is crucially used to construct the
nilpotent differential $d$ out of the Dirac operator $D$ (Lemma [\ref{the required relations}]). This means that the connected identity component of
$G$ is abelian and hence of the form $\,\mathbb{T}^m\times\mathbb{R}^n$. In our context, connectedness of the Lie group can be assumed and hence,
our Lie groups are essentially of the form $\,G=\mathbb{T}^m\times\mathbb{R}^n$ with $m+n=2k$ even.

Organization of the paper is as follows. In section $(2)$ we recall from (\cite{FGR2}) few essential definitions and a procedure to extend a $N=1$
spectral data to $N=(1,1)$ spectral data over the same noncommutative base space. Using this extension procedure we produce examples of noncommutative
K\"ahler manifolds coming from certain $C^*$-dynamical systems in section $(3)$. Our main theorem and two important corollaries are the following.
\begin{theorem}
Let $G$ be an even dimensional, connected, abelian Lie group and $(\mathcal{A}^\infty,G,\alpha,\tau)$ be a $C^*$-dynamical system equipped with a
faithful $G$-invariant trace $\tau$. If it determines a $\varTheta$-summable even spectral triple, then $\mathcal{A}^\infty$ inherits at least
$\,\prod_{j=1,\,j\,odd}^{\,dim(G)}(dim(G)-j)$ different K\"ahler structures.
\end{theorem}
\begin{corollary}
If $(\mathcal{A}^\infty,\mathbb{T}^{2k},\alpha)$ is a $C^*$-dynamical system such that the action of $\,\mathbb{T}^{2k}$ is ergodic, then
$\mathcal{A}^\infty$ inherits at least $\,\prod_{j=1,\,j\,odd}^{\,2k}(2k-j)$ different K\"ahler structures. 
\end{corollary}
\begin{corollary}
For $n$ even, the noncommutative $n$-torus $\mathcal{A}_\Theta$ satisfies the $N=(2,2)$ K\"ahler spectral data, i,e. they are noncommutative
K\"ahler manifolds. Moreover, there are at least $\prod_{j=1,\,j\,odd}^{\,n}(n-j)$ different K\"ahler structures on $\mathcal{A}_\Theta$.
\end{corollary}
In section $(4)$ we study the category of holomorphic vector bundle and show that any free module over a Hermitian spectral data is a holomorphic
vector bundle. Then we consider the particular case of noncommutative even dimensional tori, explicitly compute the space of complex differential
forms and obtain the following result.
\begin{theorem}
The category $\mathcal{H}o\ell(\mathcal{A}_\Theta)$ of holomorphic vector bundle over noncommutative even dimensional torus $\mathcal{A}_\Theta$ is
an abelian category.
\end{theorem}
At the end we explain how the earlier set-up of Polishchuk-Schwarz (\cite{PS}) for the holomorphic structure on noncommutative two-torus follows as
a special case of our general framework. We conclude this article by mentioning few important open questions.
\bigskip


\section {Preliminaries}

All algebras considered in this article will be assumed unital.

\begin{definition}\label{spectral triple}
A triple $(\mathcal{A},\mathcal{H},D)$ is called a spectral triple if
\begin{enumerate}
\item $\mathcal{A}$ is a unital associative $*$-algebra represented faithfully on the separable Hilbert space $\mathcal{H}$ by bounded operators;
\item $D$ is an unbounded self-adjoint operator acting on $\mathcal{H}$ such that for each $a\in\mathcal{A}$
\begin{enumerate}
\item the commutator $[D,a]$ extends uniquely to a bounded operator on $\mathcal{H}$,
\item $D$ has compact resolvent.
\end{enumerate}
\end{enumerate}
\end{definition}

If there is a $\mathbb{Z}_2$-grading operator on $\mathcal{H}$ such that $[\gamma,a]=0$ for all $a\in\mathcal{A}$ and $\{\gamma,D\}=0$ then the
spectral triple is called {\it even}, and otherwise {\it odd}. Note that $D$ has compact resolvent is equivalent to saying that $exp(-\varepsilon
D^2)$ is a compact operator for all $\varepsilon>0$. If $|D|^{-p}$ is in the Dixmier ideal $\mathcal{L}^{(1,\infty)}(\mathcal{H})$ then the spectral
triple is called $p$-summable.

\begin{definition}\label{N=1 spectral data}
A quadruple $(\mathcal{A},\mathcal{H},D,\gamma)$ is called a set of $N=1$ spectral data if
\begin{enumerate}
\item $\mathcal{A}$ is a unital associative $*$-algebra represented faithfully on the separable Hilbert space $\mathcal{H}$ by bounded operators;
\item $D$ is an unbounded self-adjoint operator acting on $\mathcal{H}$ such that for each $a\in\mathcal{A}$
\begin{enumerate}
\item the commutator $[D,a]$ extends uniquely to a bounded operator on $\mathcal{H}$,
\item the operator $exp(-\varepsilon D^2)$ is trace class for all $\varepsilon>0$;
\end{enumerate}
\item $\gamma$ is a $\mathbb{Z}_2$-grading on $\mathcal{H}$ such that $[\gamma,a]=0$ for all $a\in\mathcal{A}$ and $\{\gamma,D\}=0$.
\end{enumerate}
\end{definition}

\begin{remark}\rm
Observe that the $N=1$ spectral data represents a $\varTheta$-summable even spectral triple in noncommutative geometry (\cite{Con7}). In particular,
any $p$-summable even spectral triple is a $N=1$ spectral data. The associated space of differential forms, called the Dirac dga, extends the
classical de-Rham dga on manifolds to noncommutative framework (\cite{Con3}).
\end{remark}

\begin{definition}\label{N=(1,1) spectral data}
A quintuple $(\mathcal{A},\mathcal{H},d,\gamma,\star)$ is called a set of $N=(1,1)$ spectral data if
\begin{enumerate}
\item $\mathcal{A}$ is a unital associative $*$-algebra represented faithfully on the separable Hilbert space $\mathcal{H}$ by bounded operators;
\item $d$ is a densely defined closed operator on $\mathcal{H}$ such that
\begin{enumerate}
\item $d^2=0$,
\item the commutator $[d,a]$ extends uniquely to a bounded operator on $\mathcal{H}$ for each $a\in\mathcal{A}$,
\item the operator $exp(-\varepsilon\bigtriangleup)$, with $\bigtriangleup=dd^*+d^*d$, is trace class for all $\varepsilon>0$;
\end{enumerate}
\item $\gamma$ is a $\mathbb{Z}_2$-grading on $\mathcal{H}$ such that $[\gamma,a]=0$ for all $a\in\mathcal{A}$ and $\{\gamma,d\}=0$;
\item $\star$ is a unitary operator acting on $\mathcal{H}$ such that
\begin{enumerate}
\item $\star d=\zeta d^*\star\,$ for some phase $\,\zeta\in \mathbb{S}^1\subseteq \mathbb{C}$,
\item $[\star,a]=0$ for all $a\in \mathcal{A}$.
\end{enumerate}
\end{enumerate}
\end{definition}

\begin{remark}\rm
\begin{enumerate}
\item In analogy with the classical case, the operator $\star$ is called the Hodge operator.
\item As is always achievable in the classical case of manifolds, the Hodge operator can be taken to satisfy $\,\star^2=1$ and $[\star,\gamma]=0$,
and the phase $\,\zeta=-1$ (see discussion in Page $139$ in \cite{FGR2}).
\item The associated space of $N=(1,1)$ differential forms is given in (Section $[2.2.2]$ in \cite{FGR2}) and the notion of integration is described
in (Section $[2.2.3]$ in \cite{FGR2}).
\end{enumerate}
\end{remark}

\begin{definition}\label{N=(2,2) spectral data}
An octuple $(\mathcal{A},\mathcal{H},\partial,\overline{\partial},T,\overline{T},\gamma,\star)$ is called a set of $N=(2,2)$ K\"ahler spectral data if
\begin{enumerate}
\item the quintuple $(\mathcal{A},\mathcal{H},\partial+\overline{\partial},\gamma,\star)$ forms a set of $N=(1,1)$ spectral data;
\item $T,\,\overline{T}$ are bounded self-adjoint operators on $\mathcal{H}$, and $\,\partial,\overline{\partial}$ are densely defined closed operators
on $\mathcal{H}$ such that the following relations hold $:$
\begin{center}$
\begin{array}{lcl}
(a)\,\,\,\partial^2=\overline{\partial}^{\,2}=0\quad,\quad (b)\,\,\,\{\partial,\overline{\partial}\}=0\quad,\quad (c)\,\,\,[T,\overline{T}\,]=0\quad\,,\\
(d)\,\,\,[T,\partial]=\partial\quad\quad\,,\quad (e)\,\,\,[T,\overline{\partial}\,]=0\quad,\quad (f)\,\,\,[\,\overline{T},\partial]=0\quad\,,\quad (g)\,\,\,
[\,\overline{T},\overline{\partial}\,]=\overline{\partial}\,\,;
\end{array}$
\end{center}
\item $[T,a]=[\,\overline{T},a]=0\,\,\forall\,a\in\mathcal{A}$, and $\,[\partial,a],\,[\,\overline{\partial},a],\,\{\partial,[\,\overline{\partial},
a]\}$ extends uniquely to bounded operators on $\mathcal{H};$
\item the $\mathbb{Z}_2$-grading operator $\gamma$ satisfy
\begin{enumerate}
\item $\{\gamma,\partial\}=\{\gamma,\overline{\partial}\}=0$,
\item $[\gamma,T]=[\gamma,\overline{T}\,]=0;$
\end{enumerate}
\item for some phase $\,\zeta\in\mathbb{S}^1$, the Hodge operator $\star\in\mathcal{U}(\mathcal{H})$ satisfy
\begin{enumerate}
\item $\star\,\partial=\zeta\overline{\partial}^{\,*}\star\,$,
\item $\star\,\overline{\partial}=\zeta\partial^*\star\,;$
\end{enumerate}
\item the following K\"ahler conditions are satisfied
\begin{enumerate}
\item $\{\partial,\overline{\partial}^{\,*}\}=\{\,\overline{\partial},\partial^*\}=0$,
\item $\{\partial,\partial^*\}=\{\,\overline{\partial},\overline{\partial}^{\,*}\}$. 
\end{enumerate}
\end{enumerate}
\end{definition}

\begin{remark}\rm
\begin{enumerate}
\item Note that the first line in condition $(6)$ above is a consequence of the second line in classical complex geometry, but has to be imposed as
a separate condition in the noncommutative setting (\cite{FGR2}). This says that the Laplacian $\bigtriangleup=2\bigtriangleup_{\overline{\partial}}$
like in the classical case of K\"ahler manifolds.
\item An octuple $(\mathcal{A},\mathcal{H},\partial,\overline{\partial},T,\overline{T},\gamma,\star)$ satisfying conditions $(1-5)$ above is called a
{\it Hermitian spectral data} generalizing the notion of Hermitian manifolds. Condition $(6)$ is precisely the K\"ahler condition on a noncommutative
Hermitian manifold.
\item The associated space of complex differential forms is given in (Section $[2.3.2]$ in \cite{FGR2}, in particular see Propn. $[2.32]$) and the
notion of integration is described in (Section $[2.3.3]$ in \cite{FGR2}).
\end{enumerate}
\end{remark}

Defn. (\ref{N=(1,1) spectral data}) of $N=(1,1)$ spectral data has an alternative description. One can introduce two unbounded operators
\begin{center}
$\mathfrak{D}=d+d^*\quad,\quad\overline{\mathfrak{D}}=i(d-d^*)$
\end{center}
(Caution: $\overline{\mathfrak{D}}$ is not the closure of $\mathfrak{D}$) which satisfy the relations
\begin{center}
$\mathfrak{D}^2=\overline{\mathfrak{D}}^2\quad,\quad\{\mathfrak{D},\overline{\mathfrak{D}}\}=0$
\end{center}
making the notion of $N=(1,1)$ spectral data an immediate generalization of a classical $N=(1,1)$ Dirac bundle (\cite{FGR1},\cite{FGR2}). Conversely,
starting with $\,\mathfrak{D},\,\overline{\mathfrak{D}}$ satisfying the above relations, one can define
\begin{center}
$d=\frac{1}{2}(\mathfrak{D}-i\overline{\mathfrak{D}}\,)\quad,\quad d^*=\frac{1}{2}(\mathfrak{D}+i\overline{\mathfrak{D}}\,)\,.$
\end{center}
For all $\,\varepsilon>0$, the condition $\,exp(-\varepsilon(dd^*+d^*d))$ is a trace class operator is equivalent with $\,exp(-\varepsilon
\mathfrak{D}^2)$ is a trace class operator.

\begin{lemma}\label{a crucial relation involving Hodge}
If the Hodge operator satisfy $\,\star^2=1$ and $[\star,\gamma]=0$, and the phase $\,\zeta=-1$, then 
\begin{enumerate}
\item $\{\gamma,d\}=0$ if and only if $\,\{\gamma,\mathfrak{D}\}=\{\gamma,\overline{\mathfrak{D}}\}=0\,;$
\item $\star d=-d^*\star\,$ if and only if $\,\{\star,\mathfrak{D}\}=[\star,\overline{\mathfrak{D}}\,]=0\,$.
\end{enumerate}
\end{lemma}
\begin{proof}
Easy verification.
\end{proof}

Any $N=(1,1)$ spectral data gives rise to a $N=1$ spectral data over the same algebra by taking $D=d+d^*$. The converse, i,e. whether a $N=1$
spectral data can be extended to $N=(1,1)$ spectral data, is true for the classical case of manifolds (\cite{FGR1}). However, in the noncommutative
situation this is not obvious. Guided by the classical case of manifolds a procedure of extension is suggested by Fr\"{o}hlich et al. in (\cite{FGR2}),
which we discuss now.
\medskip

Let $\,\mathcal{E}$ be a finitely generated projective left module over $\mathcal{A}$ and $\,\mathcal{E}^*:=\mathcal{H}om_\mathcal{A}(\mathcal{E},
\mathcal{A})$. Clearly, $\,\mathcal{E}^*$ is also a left $\mathcal{A}$-module by the rule $\,(a\,.\,\phi)(\xi):=\phi(\xi)a^*,\,\forall\,\xi\in
\mathcal{E}$. Throughout the article, we will always write f.g.p. for notational brevity to mean finitely generated projective. In the noncommutative
situation f.g.p. module represents vector bundle over noncommutative spaces.

\begin{definition}\label{def}
A Hermitian structure on $\,\mathcal{E}$ is an $\mathcal{A}$-valued positive-definite map $\langle\,\,,\,\rangle_{\mathcal{A}}$ such that $\,:$
\begin{enumerate}
\item [(a)] $\langle\xi,\xi'\rangle_{\mathcal{A}}^*=\langle\xi',\xi\rangle_{\mathcal{A}}\,,\quad\forall\,\xi,\xi'\in\mathcal{E}$.
\item [(b)] $\langle a\,.\,\xi,b\,.\,\xi'\rangle_{\mathcal{A}}=a(\langle\xi,\xi'\rangle_{\mathcal{A}})b^*\,,\quad\forall\,\xi,\xi'\in\mathcal{E},\,\,
\forall\,a,b\in\mathcal{A}$.
\item [(c)] The map $\,g:\xi\longmapsto\Phi_\xi$ from $\,\mathcal{E}$ to $\,\mathcal{E}^*$, given by $\,\Phi_\xi(\eta)=\langle\eta,\xi\rangle_\mathcal{A}
\,,\,\forall\,\eta\in\mathcal{E}$, gives a conjugate linear left $\mathcal{A}$-module isomorphism between $\,\mathcal{E}$ and $\,\mathcal{E}^*$, i,e.
$g$ can be regarded as a metric on $\,\mathcal{E}$. This property is referred as the self-duality of $\,\mathcal{E}$.
\end{enumerate}
\end{definition}

Any free $\mathcal{A}$-module $\,\mathcal{E}_0=\mathcal{A}^n$ has a Hermitian structure on it, given by $\langle\,\xi,\eta\,\rangle_\mathcal{A}=
\sum_{j=1}^n\xi_j\eta_j^*$ for all $\xi=(\xi_1,\ldots,\xi_q)\in\mathcal{E}_0\,,\,\eta=(\eta_1,\ldots,\eta_q)\in\mathcal{E}_0$. We refer this as the
{\it canonical Hermitian structure} on $\mathcal{E}_0$. Let $\,\Omega_{D}^1(\mathcal{A})$ be the $\mathcal{A}$-bimodule $\{\sum\,a_j[D,b_j]:a_j,
b_j\in\mathcal{A}\}$ of noncommutative $1$-forms and $\,d:\mathcal{A}\longrightarrow\Omega_{D}^1(\mathcal{A})$, given by $\,a\longmapsto
[D,a]$, be the Dirac dga differential (\cite{Con3}). Note that $(da)^*=-da^*$ by convention.

\begin{definition}
Let $\,\mathcal{E}$ be a f.g.p. left module over $\mathcal{A}$ with a Hermitian structure $\langle\,\,\,,\,\,\rangle_\mathcal{A}$ on it. A compatible
connection on $\,\mathcal{E}\,$ is any $\,\mathbb{C}$-linear map $\,\nabla:\mathcal{E}\longrightarrow\Omega_{D}^1(\mathcal{A})\otimes_\mathcal{A}
\mathcal{E}$ satisfying
\begin{enumerate}
\item[(a)] $\nabla(a\xi)=a(\nabla\xi)+da\otimes\xi,\quad\forall\,\xi\in\mathcal{E},\,a\in\mathcal{A}\,;$
\item[(b)] $\langle\,\nabla\xi,\eta\,\rangle-\langle\,\xi,\nabla\eta\,\rangle=d\langle\,\xi,\eta\,\rangle_\mathcal{A}\quad\forall\,\xi,\eta\in
\mathcal{E}\,.$
\end{enumerate}
Any connection extends uniquely to a $\mathbb{C}$-linear map $\nabla:\Omega_{D}^\bullet(\mathcal{A})\otimes_\mathcal{A}\mathcal{E}\longrightarrow
\Omega_{D}^{\bullet+1}(\mathcal{A})\otimes_\mathcal{A}\mathcal{E}$ satisfying $\nabla(\omega\otimes\xi)=(-1)^{deg(\omega)}\omega\nabla(\xi)+d\omega
\otimes\xi$. The associated curvature of a connection is the $\mathcal{A}$-linear map $\,\varTheta_\nabla:\mathcal{E}\longrightarrow\Omega_{D}^2
(\mathcal{A})\otimes_\mathcal{A}\mathcal{E}$ given by the composition $\nabla\circ\nabla$.
\end{definition}

The meaning of the equality $(b)$ in $\Omega_D^1(\mathcal{A})$ is, if $\nabla(\eta)=\sum\omega_j\otimes\eta_j\in\Omega_D^1(\mathcal{A})\otimes
\mathcal{E}$, then $\langle\,\xi,\nabla\eta\,\rangle=\sum\,\langle\xi,\eta_j\rangle_\mathcal{A}\,\omega_j^*\,$.
\medskip

\textbf{A procedure to extend a $N=1$ spectral data to $N=(1,1)$ spectral data~:}\\
Start with a $N=1$ spectral data $(\mathcal{A},\mathcal{H},D,\gamma)$ equipped with a real structure $J$ (\cite{Con4},\cite{Con5}). That is,
there exists an anti-unitary operator $J$ on $\mathcal{H}$ such that
\begin{center}
$J^2=\varepsilon I\quad,\quad JD=\varepsilon^\prime DJ\quad,\quad J\gamma=\varepsilon^{\prime\prime}\gamma J\quad$
\end{center}
for some signs $\,\varepsilon,\varepsilon^\prime,\varepsilon^{\prime\prime}=\pm 1$ depending on $KO$-dimension $n\in\mathbb{Z}_8$ and satisfying
\begin{center}
$[JaJ^*,b]=[JaJ^*,[D,b]]=0\,\,\,\,\forall\,\,a,b\in\mathcal{A}\,.$
\end{center}
The real structure $J$ now enables us to equip the Hilbert space $\mathcal{H}$ with an $\mathcal{A}$-bimodule structure
\begin{center}
$a\,.\,\xi\,.\,b:=aJb^*J^*(\xi)\,.$
\end{center}
We can extend this to a right action of $\,\Omega_D^1(\mathcal{A}):=\{\sum_j\,a_j[D,b_j]:a_j,b_j\in\mathcal{A}\}$ on $\mathcal{H}$ by the rule
\begin{center}
$\xi\,.\,\omega:=J\omega^*J^*(\xi)\,.$
\end{center}
Assume that $\mathcal{H}$ contains a dense f.g.p. left $\mathcal{A}$-module $\mathcal{E}$ equipped with a Hermitian structure $\langle.,.
\rangle_\mathcal{A}$, which is stable under $J$ and $\gamma$. In particular, $\mathcal{E}$ is itself an $\mathcal{A}$-bimodule. We make
$\,\mathcal{E}\otimes_\mathcal{A}\mathcal{E}$ into an inner-product space by the following rule~:
\begin{eqnarray}\label{the inner product}
\langle\xi\otimes\eta\,,\,\xi^\prime\otimes\eta^\prime\rangle:=\langle\eta\,,\langle J\xi,J\xi^\prime\rangle_\mathcal{A}(\eta^\prime)\rangle_\mathcal{H}\,.
\end{eqnarray}
Let $\,\widetilde{\mathcal{H}}:=\overline{\mathcal{E}\otimes_\mathcal{A}\mathcal{E}}^{\langle\,\,,\,\,\rangle}$. Define the anti-linear flip operator
\begin{align*}
\Psi:\Omega_D^1(\mathcal{A})\otimes_\mathcal{A}\mathcal{E} &\longrightarrow \mathcal{E}\otimes_\mathcal{A}\Omega_D^1(\mathcal{A})\\
\omega\otimes\xi &\longmapsto J\xi\otimes\omega^*\,.
\end{align*}
It is easy to verify that $\Psi$ is well-defined and satisfies $\Psi(as)=\Psi(s)a^*,\,\,\forall\,s\in\Omega_D^1(\mathcal{A})\otimes_\mathcal{A}
\mathcal{E}$. Consider a compatible connection
\begin{center}
$\nabla:\mathcal{E}\longrightarrow\Omega_D^1(\mathcal{A})\otimes_\mathcal{A}\mathcal{E}$
\end{center}
such that $\nabla$ commutes with the grading $\gamma$ on $\,\mathcal{E}\subseteq\mathcal{H}$, i,e. $\nabla\gamma\xi=(1\otimes\gamma)\nabla\xi,\,
\forall\,\xi\in\mathcal{E}$. For each such connection $\nabla$ on $\mathcal{E}$, there is the following associated right-connection
\begin{align*}
\overline{\nabla}:\mathcal{E} &\longrightarrow \mathcal{E}\otimes_\mathcal{A}\Omega_D^1(\mathcal{A})\\
\xi &\longmapsto -\Psi(\nabla J^*\xi)
\end{align*}
Thus, we get a $\mathbb{C}$-linear map (the so called ``tensored connection'')
\begin{align*}
\widetilde\nabla:\mathcal{E}\otimes_\mathcal{A}\mathcal{E} &\longrightarrow \mathcal{E}\otimes_\mathcal{A}\Omega_D^1(\mathcal{A})\otimes_\mathcal{A}
\mathcal{E}\\
\xi_1\otimes\xi_2 &\longmapsto \overline{\nabla}\xi_1\otimes\xi_2+\xi_1\otimes\nabla\xi_2
\end{align*}
Note that $\widetilde\nabla$ is not a connection in the usual sense because of the position of $\,\Omega_D^1(\mathcal{A})$. Define the following two
$\,\mathbb{C}$-linear maps
\begin{align*}
c\,,\,\overline{c}\,:\mathcal{E}\otimes_\mathcal{A}\Omega_D^1(\mathcal{A})\otimes_\mathcal{A}\mathcal{E} &\longrightarrow \mathcal{E}
\otimes_\mathcal{A}\mathcal{E}\\
c:\xi_1\otimes\omega\otimes\xi_2 &\longmapsto \xi_1\otimes\omega\,.\,\xi_2\,,\\
\overline{c}:\xi_1\otimes\omega\otimes\xi_2 &\longmapsto \xi_1\,.\,\omega\otimes\gamma\xi_2\,.
\end{align*}
Now, introduce the following densely defined unbounded operators on $\widetilde{\mathcal{H}}$ $$\mathfrak{D}:=c\circ\widetilde{\nabla}\quad,\quad
\overline{\mathfrak{D}}:=\overline{c}\circ\widetilde{\nabla}$$ (Caution: $\overline{\mathfrak{D}}$ is not the closure of $\mathfrak{D}$). In order
to obtain a set of $N=(1,1)$ spectral data on $\mathcal{A}$, one has to find a specific connection $\nabla$ on a suitable dense Hermitian f.g.p. left
$\mathcal{A}$-module $\,\mathcal{E}$ such that
\begin{itemize}
\item[(a)] The operators $\mathfrak{D}$ and $\overline{\mathfrak{D}}$ become essentially self-adjoint on $\widetilde{\mathcal{H}}$,
\item[(b)] The relations $\,\mathfrak{D}^2=\overline{\mathfrak{D}}^2$ and $\{\mathfrak{D},\overline{\mathfrak{D}}\}=0$ are satisfied.
\end{itemize}
The $\mathbb{Z}_2$-grading on $\widetilde{\mathcal{H}}$ is simply the tensor product grading $\,\widetilde{\gamma}:=\gamma\otimes\gamma$, and the
Hodge operator is taken to be $\,\star:=1\otimes\gamma$ (In \cite{FGR2}, this is mistakenly taken as $\,\star=\gamma\otimes 1$). The sextuple
$(\mathcal{A},\widetilde{\mathcal{H}},\mathfrak{D},\overline{\mathfrak{D}},\widetilde{\gamma},\star)$ is a candidate of $N=(1,1)$ spectral data
extending the $N=1$ spectral data $(\mathcal{A},\mathcal{H},D,\gamma)$. This Hodge operator additionally satisfies $\,\star^2=1$ and $[\star,\gamma]=0$.
Hence, Lemma (\ref{a crucial relation involving Hodge}) holds for this extension procedure.
\medskip

Recently, behaviour of this extension procedure under tensor product has been investigated in (\cite{G}). Apart from the classical case of manifolds,
existence of such suitable connection $\nabla$ is known for the noncommutative $2$-torus and the fuzzy $3$-sphere (target space of the $SU(2)$ WZW
model \cite{Wit2}). However, the general case remains open. In the next section we will see a class of examples coming from certain $C^*$-dynamical
systems satisfying this extension procedure.
\bigskip


\section {K\"ahler structure on certain \texorpdfstring{$C^*$}{Lg}-dynamical systems}

\begin{definition}
A $\,C^*$-dynamical system is a tuple $(\mathcal{A},G,\alpha)$ where $\mathcal{A}$ is a unital $\,C^*$-algebra, $G$ is a real Lie group and $\,\alpha:
G\rightarrow Aut(\mathcal{A})$ is a strongly continuous group homomorphism $($i,e. for all $a\in\mathcal{A}$, the map $g\mapsto\alpha_g(a)$ is
continuous$)$.
\end{definition}

We will work with $C^*$-dynamical systems $(\mathcal{A},G,\alpha)$ equipped with a faithful $G$-invariant trace $\tau$, i,e. $\tau(\alpha_g(a))=
\tau(a)$ for all $g\in G$. This is in line with (\cite{Con1},\cite{CR},\cite{Rfl3}). Note that if the Lie group is compact and the action is ergodic
then the unique $G$-invariant state is a faithful trace on $\mathcal{A}$ (\cite{HLS}). We say that $a\in\mathcal{A}$ is smooth if the map $g\mapsto
\alpha_g(a)$ is in $C^\infty(G,\mathcal{A})$. The involutive algebra $\mathcal{A}^{\infty}=\{a\in\mathcal{A}:\,a\,\,\hbox{is smooth}\}$ is a norm
dense subalgebra of $\mathcal{A}$, called the smooth subalgebra. Note that this is unital as well. One crucial property enjoyed by this subalgebra
is that it is closed under the holomorphic function calculus inherited from the ambient $C^*$-algebra $\mathcal{A}$ (\cite{GVF}). Henceforth, we
will always work with the smooth subalgebra $\mathcal{A}^{\infty}$ and denote it simply by $\mathcal{A}$ for notational brevity.

To begin with we recall a result in (\cite{DD}) which provides an explicit set of generators of the irreducible representations of $\,\mathbb{C}l(n)$
for all $n$, together with an explicit involution $J$ and (if $n$ is even) a grading operator $\gamma$. This is summarized in below.

\begin{proposition}[\cite{DD}]\label{Clifford representation}
Consider a positive integer $n$ and an irreducible representation of $\,\mathbb{C}l(n)$ on a vector space $\mathbb{V}$. Up to unitary equivalence,
it is determined by $n$ many matrices $\gamma_j$ such that
\begin{center}
$\gamma_j^*=-\gamma_j\quad,\quad\gamma_j\gamma_k+\gamma_k\gamma_j=-2\delta_{jk}$.
\end{center}
If $n$ is even, there is a $\mathbb{Z}_2$ grading operator $\gamma_V$ satisfying $\gamma_V\gamma_j=-\gamma_j\gamma_V$ for all $j=1,\ldots,n$.
Moreover, there is an explicit anti-isometry $J_V\,\,($charge conjugation$)$ satisfying
\begin{center}
$(J_V)^{2}=\varepsilon\quad,\quad J_V\gamma_j=\varepsilon^\prime\gamma_jJ_V\quad,\quad J_V\gamma_V=\varepsilon^{\prime\prime}\gamma_VJ_V$
\end{center}
for some signs $\,\varepsilon,\varepsilon^\prime,\varepsilon^{\prime\prime}\in\{1,-1\}$ depending on $n$ modulo $8\,:$
\medskip

\begin{center}
\begin{tabular}{|c| c c c c| c c c c|}
\hline
$n$ & $0$ & $2$ & $4$ & $6$ & $1$ & $3$ & $5$ & $7$ \\ [0.8ex]
\hline
$\varepsilon$ & $+$ & $-$ & $-$ & $+$ & $+$ & $-$ & $-$ & $+$\\
$\varepsilon^\prime$  & $+$ & $+$ & $+$ & $+$ & $-$ & $+$ & $-$ & $+$\\
$\varepsilon^{\prime\prime}$ & $+$ & $-$ & $+$ & $-$ & $\,$ & $\,$ & $\,$ & $\,$\\
\hline
\end{tabular}
\end{center}
\end{proposition}

\textbf{Candidate of a $N=1$ spectral data associated with $C^*$-dynamical systems~:}\\
Let $(\mathcal{A},G,\alpha,\tau)$ be a $\,C^*$-dynamical system equipped with a $G$-invariant faithful trace $\tau$. Let $\,dim(G)=n$ and $N=2^{\lfloor
n/2\rfloor}$. Let $\{X_1,\ldots,X_n\}$ be a basis of the Lie algebra $\mathfrak{g}$ of the Lie group $G$. Letting $\,\mathcal{H}=L^2(\mathcal{A},
\tau)$ the G.N.S Hilbert space, we obtain a covariant representation of $(\mathcal{A},G,\alpha)$ on $\mathcal{H}$. Note that there is a bijective
correspondence between covariant representations of $(\mathcal{A},G,\alpha)$ and non-degenerate $\star$-representations of $\mathcal{A}\rtimes_\alpha
G$. We obtain the following densely defined unbounded symmetric operator $$D:=\sum_{j=1}^n\,\partial_j\otimes\gamma_j$$ acting on $\,\widetilde{
\mathcal{H}}=L^2(\mathcal{A},\tau)\otimes\mathbb{C}^N$, where $\partial_j(a):=\frac{d}{dt}|_{t=0}\,\alpha_{exp(tX_j)}(a)$ and $\gamma_j$ as in the
above Proposition. Note that the map $\,\partial:\mathfrak{g}\longrightarrow Der(\mathcal{A})$ given by $\,X_j\longmapsto\partial_j$, where
$Der(\mathcal{A})$ is the Lie algebra of derivations on $\mathcal{A}$, is a Lie algebra homomorphism i,e. $[\partial_j,\partial_\ell]=
\partial_{[j,\ell]}$. Moreover, there is always a real structure $J=J_0\otimes J_N$, where $J_0$ is the anti-linear operator $a\mapsto a^*$ and
$J_N=J_V$ as in the above Proposition (\ref{Clifford representation}), satisfying
\begin{center}
$J^2=\varepsilon I\quad,\quad JD=\varepsilon^\prime DJ\quad,\quad J\gamma_N=\varepsilon^{\prime\prime}\gamma_N J$
\end{center}
(if grading $\gamma_N=\gamma_V$ exists) for some signs $\,\varepsilon,\varepsilon^\prime,\varepsilon^{\prime\prime}=\pm 1$ depending on $n\in
\mathbb{Z}_8$ and satisfying the above mentioned table. We also have
\begin{center}
$[JaJ^*,b]=[JaJ^*,[D,b]]=0\,\,\,\,\forall\,\,a,b\in\mathcal{A}\,.$
\end{center}
It is known that $D$, defined above, admits a self-adjoint extension (\cite{Gab}). But the summability and compactness of the resolvent of $D$ is
not guaranteed. So, if $D$ is essentially self-adjoint with compact resolvent and gives the $\varTheta$-summability (note that any finitely summable
spectral triple is $\varTheta$-summable \cite{Con7}), then we obtain a $\varTheta$-summable even spectral triple $(\mathcal{A},\widetilde{\mathcal{H}},
D,\gamma_N)$ if $n$ is even; otherwise odd spectral triple $(\mathcal{A},\widetilde{\mathcal{H}},D)$ if $n$ is odd. However, existence of such a
self-adjoint extension of $D$ is an intricate question and that is why we only get a candidate of a $N=1$ spectral data.

\begin{remark}\rm
\begin{itemize}
\item[(a)] It is known that if the Lie group $G$ is compact then $D$, defined above, is essentially self-adjoint (Propn. $[4.1]$ in \cite{Gab}).
\item[(b)] If the Lie group $G$ is compact and acts ergodically then we obtain a $dim(G)$-summable (and hence $\varTheta$-summable) spectral triple
(Thm. $[5.4]$ in \cite{Gab}), independent of the choice of the Lie algebra basis. This is the case for the noncommutative $n$-torus $\mathcal{A}_\Theta$.
\item[(c)] Compactness and ergodicity is a sufficient condition only. Recall the case of quantum Heisenberg manifolds (\cite{Rfl2}) where the Lie
group acting is noncompact namely, the Heisenberg group. It is known (\cite{CSi},\cite{CG2}) that one gets an honest $3$-summable spectral
triple in this case for a suitable choice of the Heisenberg Lie algebra basis.
\item[(d)] There is no characterization of $C^*$-dynamical systems known yet which gives genuine finite or $\varTheta$-summable spectral triples
by the above discussed method.
\end{itemize}
\end{remark}

Guided by these we start with a $C^*$-dynamical system $(\mathcal{A},G,\alpha)$ equipped with a $G$-invariant faithful trace $\tau$, where $G$ is an
even dimensional, connected, abelian Lie group, so that the candidate, discussed above, determines an honest $N=1$ spectral data $(\mathcal{A},
\mathcal{H},D,\sigma)$ with $\,\sigma$ the $\mathbb{Z}_2$ grading. Since the Lie algebra $\,\mathfrak{g}$ is abelian, i,e. $[X_j,X_\ell]=0$ for all
$\,j,\ell\in\{1,\ldots,dim(G)\}$, we have $[\partial_j,\partial_\ell]=0$. Our first objective is to show that this $N=1$ spectral data always extends
to $N=(1,1)$ spectral data over $\mathcal{A}$ by the procedure of extension discussed in Section $(2)$. Then we produce K\"ahler structures on
$\mathcal{A}$. The key ingredient is the Grassmannian connection as we shall see. Let $dim(G)=2k$ and $N=2^{\lfloor dim(G)/2\rfloor}=2^k$. Consider
the dense finitely generated free left $\mathcal{A}$-module $\,\mathcal{E}:=\mathcal{A}\otimes\mathbb{C}^N\subseteq\mathcal{H}=L^2(\mathcal{A},\tau)
\otimes\mathbb{C}^N$ equipped with the canonical Hermitian structure. Clearly, $\mathcal{E}$ stable under the real structure $J=J_0\otimes J_N$ and
the grading operator $\,\sigma$. Consider the following $\mathbb{C}$-linear map
\begin{align*}
\nabla:\mathcal{E} &\longrightarrow \Omega_D^1(\mathcal{A})\otimes_\mathcal{A}\mathcal{E}\\
\xi &\longmapsto (d\xi_1,\ldots\ldots,d\xi_N)
\end{align*}
for $\xi=(\xi_1,\ldots,\xi_N)\in\mathcal{E}$, where $d:\mathcal{A}\longrightarrow\Omega_D^1(\mathcal{A})$, given by $a\mapsto[D,a]$, is the Dirac dga
differential. This is the Grassmannian connection on the free left $\mathcal{A}$-module $\mathcal{E}$ and is easily seen to be compatible with the
canonical Hermitian structure given by $\langle\xi,\eta\rangle:=\sum_{j=1}^N\,\xi_j\eta_j^*$. We fix the standard canonical free $\mathcal{A}$-module
basis $\{e_1,\ldots,e_N\}$ of $\,\mathcal{E}=\mathcal{A}\otimes\mathbb{C}^N$. By abuse of notation, same denotes the canonical linear basis of
$\,\mathbb{C}^N$ if no confusion arise.

\begin{lemma}
The Grassmannian connection $\nabla:\mathcal{E}\longrightarrow\Omega_D^1(\mathcal{A})\otimes_\mathcal{A}\mathcal{E}$ and its associated right
connection $\overline{\nabla}:\mathcal{E}\longrightarrow\mathcal{E}\otimes_\mathcal{A}\Omega_D^1(\mathcal{A})$ satisfy $\nabla e_j=\overline{\nabla}
e_j=0\,\,\forall\,j\in\{1,\ldots,N\}$, and it commutes with the $\mathbb{Z}_2$-grading $\,\sigma$.
\end{lemma}
\begin{proof}
Clearly, $\nabla e_j=0\,\,\forall\,j\in\{1,\ldots,N\}$ by its definition. Note that, $\overline{\nabla}(e_j)=-\Psi(\nabla J^*e_j)$. Since,
$J=J_0\otimes J_N,\,e_j=1\otimes(0,\ldots,1,\ldots,0)\in\mathcal{A}\otimes\mathbb{C}^N$ and $J_N^*=\varepsilon J_N$, we get
\begin{eqnarray*}
J^*e_j & = & \varepsilon 1\otimes J_N(0,\ldots,1,\ldots,0)^T\\
& = & \varepsilon 1\otimes ((J_N)_{1j},\ldots,(J_N)_{Nj})^T\\
& = & \sum_{k=1}^N\varepsilon(J_N)_{kj}e_k\,.
\end{eqnarray*}
Since, $(J_N)_{kj}$ are scalars  for all $k$ and $\nabla$ is $\mathbb{C}$-linear map satisfying $\nabla e_k=0$, our claim follows. Finally,
commutation of $\nabla$ with the $\mathbb{Z}_2$-grading operator is easy to observe.
\end{proof}

Note that any element of $\,\mathcal{E}\otimes_\mathcal{A}\mathcal{E}$ of the form $\,ae_i\otimes be_j$  can be written as $\,e_i.(J^*a^*J)\otimes
be_j$ (recall the right $\mathcal{A}$-module structure on $\mathcal{E}$), and since the tensor is over $\mathcal{A}$, this is same as $\,e_i\otimes
ce_j$ for some $c\in\mathcal{A}$. Now, for any arbitrary element $e_i\otimes a_{ij}e_j$ of $\,\mathcal{E}\otimes_\mathcal{A}\mathcal{E}$, the
tensored connection $\widetilde{\nabla}$ becomes
\begin{eqnarray*}
\widetilde{\nabla}(e_i\otimes a_{ij}e_j) & = & \overline{\nabla}e_i\otimes a_{ij}e_j+e_i\otimes\nabla(a_{ij}e_j)\\
& = & e_i\otimes da_{ij}\otimes e_j\,.
\end{eqnarray*}
So, we have
\begin{eqnarray}\label{two Dirac operators}
\mathfrak{D}(e_i\otimes a_{ij}e_j)\,:=\,c\circ\widetilde{\nabla}(e_i\otimes a_{ij}e_j) & = & e_i\otimes (da_{ij})\,.\,e_j\\\nonumber
\overline{\mathfrak{D}}(e_i\otimes a_{ij}e_j)\,:=\,\bar{c}\circ\widetilde{\nabla}(e_i\otimes a_{ij}e_j) & = & e_i\,.\,(da_{ij})\otimes\sigma e_j
\end{eqnarray}
where, $\,\sigma$ is the $\mathbb{Z}_2$-grading operator.

\begin{proposition}\label{an isomorphism of Hilbert spaces}
The Hilbert space $\,\overline{\mathcal{E}\otimes_\mathcal{A}\mathcal{E}}$ is $\,L^2(\mathcal{A},\tau)^{N^2}=L^2(\mathcal{A},\tau)\otimes\mathbb{C}^{N^2}$.
\end{proposition}
\begin{proof}
Since $\mathcal{E}=\mathcal{A}\otimes\mathbb{C}^N$, we have $\mathcal{E}\otimes_\mathcal{A}\mathcal{E}$ is isomorphic with $\mathcal{A}\otimes
\mathbb{C}^{N^2}$. Because $\mathcal{E}$ has the canonical Hermitian structure on it, from the inner-product defined in  (\ref{the inner product}),
it follows that
\begin{eqnarray*}
\langle\xi\otimes\eta\,,\,\xi^\prime\otimes\eta^\prime\rangle & = & \langle\eta\,,\langle J\xi,J\xi^\prime\rangle_\mathcal{A}(\eta^\prime)\rangle_\mathcal{H}\\
& = & \sum_{\ell,j}\langle\eta_\ell\,,\xi_j^*\xi_j^\prime\eta_\ell^\prime\rangle\\
& = & \sum_{\ell,j}\tau\left(\eta_\ell^*\xi_j^*\xi_j^\prime\eta_\ell^\prime\right)\,.
\end{eqnarray*}
This is precisely the inner-product on $\mathcal{A}\otimes\mathbb{C}^{N^2}$ given by the inner-product $\langle a,b\rangle:=\tau(a^*b)$ on
$\mathcal{A}$ and the usual inner-product on $\,\mathbb{C}^{N^2}$. The completion is the Hilbert space $L^2(\mathcal{A},\tau)^{N^2}=L^2
(\mathcal{A},\tau)\otimes\mathbb{C}^{N^2}$, and this concludes the proof.
\end{proof}

\begin{lemma}\label{symmetric operator}
$\mathfrak{D}$ and $\overline{\mathfrak{D}}$ are densely defined symmetric operators acting on the Hilbert space $\,\overline{\mathcal{E}
\otimes_\mathcal{A}\mathcal{E}}$.
\end{lemma}
\begin{proof}
We have
\begin{eqnarray*}
&   & \langle\mathfrak{D}(e_i\otimes a_{ij}e_j),e_m\otimes a_{m\ell}e_\ell\rangle-\langle e_i\otimes a_{ij}e_j,\mathfrak{D}(e_m\otimes a_{m\ell}e_\ell)\rangle\\
& = & \langle e_i\otimes (da_{ij}).e_j,e_m\otimes a_{m\ell}e_\ell\rangle-\langle e_i\otimes a_{ij}e_j,e_m\otimes (da_{m\ell}).e_\ell\rangle\\
& = & \langle (da_{ij}).e_j,\langle e_i,e_m\rangle_\mathcal{A}(a_{m\ell}e_\ell)\rangle-\langle a_{ij}e_j,\langle e_i,e_m\rangle_\mathcal{A}(da_{m\ell}).e_\ell\rangle\\
& = & \delta_{im}\left(\langle (da_{ij}).e_j,a_{m\ell}e_\ell\rangle-\langle a_{ij}e_j,(da_{m\ell}).e_\ell\rangle\right)\\
& = & \langle (da_{ij}).e_j,a_{i\ell}e_\ell\rangle-\langle a_{ij}e_j,(da_{i\ell}).e_\ell\rangle\\
& = & \left\langle\,\sum_{r=1}^{2k}\partial_r(a_{ij})\otimes(\gamma_{r1j},\ldots,\gamma_{rNj})\,,\,a_{i\ell}e_\ell\right\rangle-\left\langle a_{ij}
e_j\,,\,\sum_{r=1}^{2k}\partial_r(a_{i\ell})\otimes(\gamma_{r1\ell},\ldots,\gamma_{rN\ell})\right\rangle\\
& = & \sum_{r=1}^{2k}\,\langle(\partial_r(a_{ij})\gamma_{r1j},\ldots,\partial_r(a_{ij})\gamma_{rNj}),(0,\ldots,a_{i\ell},\ldots,0)\rangle\\
&   & \quad-\langle(0,\ldots,a_{ij},\ldots,0),(\partial_r(a_{i\ell})\gamma_{r1\ell},\ldots,\partial_r(a_{i\ell})\gamma_{rN\ell})\rangle\\
& = & \sum_{r=1}^{2k}\tau((\partial_r(a_{ij})\gamma_{r\ell j})^*a_{i\ell})-\tau(a_{ij}^*\partial_r(a_{i\ell})\gamma_{rj\ell})\\
& = & \sum_{r=1}^{2k}\tau\left(\partial_r(a_{ij}^*)\overline{\gamma_{r\ell j}}a_{i\ell}\right)+\tau\left(a_{ij}^*\partial_r(a_{i\ell}\overline{\gamma_{r\ell j}})\right)\\
& = & \sum_{r=1}^{2k}\tau\left(\partial_r(a_{ij}^*\overline{\gamma_{r\ell j}}a_{i\ell})\right)
\end{eqnarray*}
Here, we are using the fact that for all $\,r\in\{1,\ldots,2k\},\,\gamma_r^*=-\gamma_r$. Hence, $\overline{(\gamma_r)_{\ell j}}=-(\gamma_r)_{j\ell}$.
Now, for any $a\in\mathcal{A}$,
\begin{center}
$\tau(\partial_r(a))\,=\,\tau\left(\frac{d}{dt}|_{t=0}\,\alpha_{exp(tX_r)}(a)\right)\,=\,0$
\end{center}
for all $r\in\{1,\ldots,2k\}$, because $\tau$ is a $G$-invariant trace. This proves that $\mathfrak{D}$ is a symmetric operator. Similarly, one can
show for $\overline{\mathfrak{D}}$.
\end{proof}

\begin{proposition}\label{self-adjoint operator}
Both $\mathfrak{D}$ and $\overline{\mathfrak{D}}$ are unbounded self-adjoint operators acting on the Hilbert space $\,\overline{\mathcal{E}
\otimes_\mathcal{A}\mathcal{E}}=L^2(\mathcal{A},\tau)^{N^2}$.
\end{proposition}
\begin{proof}
Observe that for any $\,\xi=e_i\otimes a_{ij}e_j\in\mathcal{E}\otimes_\mathcal{A}\mathcal{E}$, we can write $$\xi=\underbrace{(0,\ldots,
\underbrace{(0,\ldots,a_{ij},\ldots,0)}_{i-th\,\,place\,,\,N\,\,tuple},\ldots,0)}_{N\,\,tuple}\,\in\mathcal{A}^{N^2}$$ and hence,
\begin{eqnarray*}
\mathfrak{D}(\xi) & = & (0,\ldots,\underbrace{\sum_{r=1}^{2k}\partial_r(a_{ij})\otimes\gamma_r(e_j)}_{\in\,L^2(\mathcal{A},\tau)^N},\ldots,0)\,\in
L^2(\mathcal{A},\tau)^{N^2}\\
& = & e_i\otimes D(a_{ij}e_j)
\end{eqnarray*}
Now,
\begin{eqnarray*}
\overline{\mathfrak{D}}(\xi) & = & e_i\,.\,da_{ij}\otimes\sigma e_j\\
& = & -\varepsilon Jd(a_{ij}^*)Je_i\otimes\sigma e_j\\
& = & -\varepsilon^\prime\left(\sum_{r=1}^{2k}\partial_r(a_{ij})\otimes\gamma_r(e_i)\right)\otimes\sigma e_j\\
& = & -\varepsilon^\prime D(a_{ij}e_i)\otimes\sigma e_j
\end{eqnarray*}
and observe that
\begin{eqnarray*}
e_i\otimes a_{ij}e_j & = & e_i\,.\,a_{ij}\otimes e_j\\
& = & Ja_{ij}^*J^*e_i\otimes e_j\\
& = & \varepsilon Ja_{ij}^*(1\otimes J_Ne_i)\otimes e_j\\
& = & \varepsilon(a_{ij}\otimes J_N^2e_i)\otimes e_j\\
& = & a_{ij}e_i\otimes e_j\,.\\
\end{eqnarray*}
Since, $\overline{\mathcal{E}\otimes_\mathcal{A}\mathcal{E}}\cong L^2(\mathcal{A},\tau)^{N^2}$ (Propn. [\ref{an isomorphism of Hilbert spaces}]),
we see that the operator $\mathfrak{D}$ is of the form $1_N\otimes D$ and the operator $\overline{\mathfrak{D}}$ is of the form $-\varepsilon^\prime
D\otimes\sigma$, both acting on $L^2(\mathcal{A},\tau)^{N^2}$. That is,
\begin{center}
$\mathfrak{D}=\sum_{j=1}^{2k}\partial_j\otimes 1_N\otimes\gamma_j\quad$ and $\quad\overline{\mathfrak{D}}=-\varepsilon^\prime\sum_{j=1}^{2k}
\partial_j\otimes\gamma_j\otimes\sigma$
\end{center}
acting on $L^2(\mathcal{A},\tau)^{N^2}\cong L^2(\mathcal{A},\tau)\otimes\mathbb{C}^N\otimes\mathbb{C}^N$. Since, we have assumed that the
$C^*$-dynamical system $(\mathcal{A},G,\alpha,\tau)$ gives us an honest $N=1$ spectral data $\left(\mathcal{A},L^2(\mathcal{A},\tau)\otimes\mathbb{C}^N,
D=\sum_{j=1}^{2k}\partial_j\otimes\gamma_j\right)$; $D$ is self-adjoint on $\mathcal{H}=L^2(\mathcal{A},\tau)\otimes\mathbb{C}^N$. This proves the
self-adjointness of $\,\mathfrak{D}$ and $\overline{\mathfrak{D}}$.
\end{proof}

\begin{remark}\rm\label{why to keep epsilon}
Since we are dealing with even dimensional Lie groups, $\varepsilon^\prime=+1$ by the table mentioned in Propn. (\ref{Clifford representation}).
However, we intend not to discard $\varepsilon^\prime$ in the expression of $\overline{\mathfrak{D}}$ for the time being for a specific reason. This
will be explained towards the end of this section before our final theorem.
\end{remark}

\begin{lemma}\label{the required relations}
We have the relations $\,\mathfrak{D}^2=\overline{\mathfrak{D}}^{\,2}$ and $\{\mathfrak{D},\overline{\mathfrak{D}}\}=0$.
\end{lemma}
\begin{proof}
Since $\,\sigma$ is a $\mathbb{Z}_2$-grading operator on $(\mathcal{A},\mathcal{H},D)$, we have $\{D,\sigma\}=0$. This gives $\{\mathfrak{D},
\overline{\mathfrak{D}}\}=0$. Now,
\begin{eqnarray*}
\mathfrak{D}^2 & = & -\left(\sum_{r=1}^{2k}\partial_r^2\right)\otimes 1_N\otimes 1_N+\sum_{i<j}[\partial_i,\partial_j]\otimes 1_N\otimes\gamma_i\gamma_j\\
& = & -\left(\sum_{r=1}^{2k}\partial_r^2\right)\otimes 1_N\otimes 1_N+\sum_{i<j}\partial_{[i,j]}\otimes 1_N\otimes\gamma_i\gamma_j\\
& = & -\left(\sum_{r=1}^{2k}\partial_r^2\right)\otimes 1_N\otimes 1_N
\end{eqnarray*}
because $\,\mathfrak{g}$ is abelian. One gets exactly equal expression for $\overline{\mathfrak{D}}^{\,2}$. Hence, as operators on $L^2(\mathcal{A},
\tau)^{N^2}$ we get $\,\mathfrak{D}^2=\overline{\mathfrak{D}}^2$.
\end{proof}

\begin{remark}\rm
This is the place where we need $\,\mathfrak{g}\,$ is abelian to conclude that $\,\mathfrak{D}^2=\overline{\mathfrak{D}}^{\,2}$. Without this we can
not have $\,d^{\,2}=0$, where $\,d=\frac{1}{2}(\mathfrak{D}-i\,\overline{\mathfrak{D}})$.
\end{remark}

\begin{lemma}\label{required traciality}
We have the following :
\begin{itemize}
\item[(i)] For all $\,a\in\mathcal{A},\,[d,a]$ extends to a bounded operator acting on the Hilbert space $\,\overline{\mathcal{E}\otimes_\mathcal{A}
\mathcal{E}}$, where $\,d=\frac{1}{2}(\mathfrak{D}-i\,\overline{\mathfrak{D}})$.
\item[(ii)] $exp(-\varepsilon\mathfrak{D}^2)$ is a trace class operator for all $\,\varepsilon>0$.
\end{itemize}
\end{lemma}
\begin{proof}
Both these facts follow from our assumption that the $C^*$-dynamical system $(\mathcal{A},G,\alpha,\tau)$ gives us an honest $N=1$ spectral data
$\left(\mathcal{A},\mathcal{H}=L^2(\mathcal{A},\tau)\otimes\mathbb{C}^N,D=\sum_{j=1}^{2k}\partial_j\otimes\gamma_j\right)$, and the explicit expressions
of $\mathfrak{D}$ and $\overline{\mathfrak{D}}$ in Propn. (\ref{self-adjoint operator}). Note that $\,Tr(exp(-\varepsilon\mathfrak{D}^2))=NTr(exp
(-\varepsilon D^2))$ for all $\,\varepsilon>0$.
\end{proof}

\begin{proposition}\label{N=(1,1) from dynamical system}
Let $G$ be an even dimensional, connected, abelian Lie group and $(\mathcal{A},G,\alpha,\tau)$ be a $C^*$-dynamical system equipped with a faithful
$G$-invariant trace $\tau$. If it determines a $N=1$ spectral data $(\mathcal{A},\mathcal{H},D,\sigma)$ then it always extends to $N=(1,1)$ spectral
data over $\mathcal{A}$.
\end{proposition}
\begin{proof}
Combining Propn. (\ref{self-adjoint operator}) and Lemma (\ref{the required relations}\,,\,\ref{required traciality}) we see that the only
remaining part is to produce a $\,\mathbb{Z}_2$-grading and a Hodge operator. We have two self-adjoint unitaries $\gamma:=1\otimes\sigma\otimes 1_N$
and $\gamma^\prime:=1\otimes 1_N\otimes\sigma$ acting on the Hilbert space $L^2(\mathcal{A},\tau)^{N^2}=L^2(\mathcal{A},\tau)\otimes\mathbb{C}^N
\otimes\mathbb{C}^N$, satisfying
\begin{center}
$\{\mathfrak{D},\gamma^\prime\}=\{\overline{\mathfrak{D}},\gamma\}=0\quad,\quad[\mathfrak{D},\gamma]=[\,\overline{\mathfrak{D}},\gamma^\prime]=0\,\,.$
\end{center}
The $\mathbb{Z}_2$-grading is obtained by taking $\widetilde{\gamma}:=\gamma\gamma^\prime=1\otimes\sigma\otimes\sigma$. Clearly, $\{\widetilde{\gamma},
\mathfrak{D}\}=\{\widetilde{\gamma},\overline{\mathfrak{D}}\}=0$. Finally, the Hodge operator is given by $\,\star:=\gamma^\prime=1\otimes 1_N\otimes
\sigma$ acting on $\,\overline{\mathcal{E}\otimes_\mathcal{A}\mathcal{E}}=L^2(\mathcal{A},\tau)^{N^2}$, and it satisfies $\,\{\star,\mathfrak{D}\}=
[\star,\overline{\mathfrak{D}}\,]=0$ for the phase $\zeta=-1$. This concludes the proof in view of Lemma (\ref{a crucial relation involving Hodge}).
\end{proof}

\begin{lemma}\label{An operator related to d}
The following bounded self-adjoint operator
\begin{align*}
\mathcal{T}:L^2(\mathcal{A},\tau)\otimes\mathbb{C}^N\otimes\mathbb{C}^N &\longrightarrow L^2(\mathcal{A},\tau)\otimes\mathbb{C}^N\otimes\mathbb{C}^N\\
\mathcal{T} &:= \sum_{j=1}^{2k}\,\frac{i\varepsilon^\prime}{2}1\otimes\gamma_j\otimes\gamma_j\sigma
\end{align*}
commutes with all elements of $\mathcal{A}\subseteq\mathcal{B}\left(L^2(\mathcal{A},\tau)^{N^2}\right)$ and $[\mathcal{T},d\,]=d$, where $\,d=
\frac{1}{2}(\mathfrak{D}-i\overline{\mathfrak{D}})$.
\end{lemma}
\begin{proof}
Recall that $L^2(\mathcal{A},\tau)^{N^2}=\overline{\mathcal{E}\otimes_\mathcal{A}\mathcal{E}}$ (Propn. [\ref{an isomorphism of Hilbert spaces}]) and
$\mathcal{A}$ is represented on $\overline{\mathcal{E}\otimes_\mathcal{A}\mathcal{E}}$ by its left action on $\,\mathcal{E}$. Clearly, $\mathcal{T}$
then commutes with $\mathcal{A}\subseteq\mathcal{B}\left(L^2(\mathcal{A},\tau)^{N^2}\right)$. Recall the expressions for $\mathfrak{D}$ and
$\overline{\mathfrak{D}}$ from Propn. (\ref{self-adjoint operator}). We have the following
\begin{eqnarray*}
&   & [\,-\varepsilon^\prime\mathcal{T}\,,\,\mathfrak{D}-i\overline{\mathfrak{D}}\,]\\
& = & \sum_{j=1}^{2k}\,\left[\frac{1}{2i}1\otimes\gamma_j\otimes\gamma_j\sigma\,,\,\mathfrak{D}-i\overline{\mathfrak{D}}\,\right]\\
& = & \sum_{j=1}^{2k}\,\frac{1}{2i}[1\otimes\gamma_j\otimes\gamma_j\sigma\,,\,1\otimes D]+\frac{\varepsilon^\prime}{2}[1\otimes\gamma_j\otimes
\gamma_j\sigma\,,\,D\otimes\sigma]\\
& = & \sum_{j,r=1}^{2k}\,\frac{1}{2i}[1\otimes\gamma_j\otimes\gamma_j\sigma\,,\,\partial_r\otimes1_N\otimes\gamma_r]+\frac{\varepsilon^\prime}{2}[1\otimes
\gamma_j\otimes\gamma_j\sigma\,,\,\partial_r\otimes\gamma_r\otimes\sigma]\\
& = & \sum_{j,r=1}^{2k}\,\frac{1}{2i}\partial_r\otimes\gamma_j\otimes(\gamma_j\sigma\gamma_r-\gamma_r\gamma_j\sigma)+\frac{\varepsilon^\prime}{2}
\partial_r\otimes(\gamma_j\gamma_r+\gamma_r\gamma_j)\otimes\gamma_j\\
& = & \sum_{j,r=1}^{2k}\,\frac{1}{2i}\partial_r\otimes\gamma_j\otimes(-\gamma_j\gamma_r-\gamma_r\gamma_j)\sigma+\frac{\varepsilon^\prime}{2}
\partial_r\otimes(\gamma_j\gamma_r+\gamma_r\gamma_j)\otimes\gamma_j\\
& = & \sum_{j=1}^{2k}\,\frac{1}{i}\partial_j\otimes\gamma_j\otimes\sigma-\varepsilon^\prime\partial_j\otimes 1_N\otimes\gamma_j\\
& = & \frac{1}{i}(D\otimes\sigma)-\varepsilon^\prime(1\otimes D)\\
& = & i\varepsilon^\prime\overline{\mathfrak{D}}-\varepsilon^\prime\mathfrak{D}\\
& = & -\varepsilon^\prime(\mathfrak{D}-i\overline{\mathfrak{D}})\,.
\end{eqnarray*}
Hence, for $\,d=\frac{1}{2}(\mathfrak{D}-i\overline{\mathfrak{D}})$ we see that $[\mathcal{T},d\,]=d$.
\end{proof}

\begin{proposition}\label{sufficient condition for complex structure}
If there exists a skew-adjoint matrix $\widetilde{\mathcal{I}}\in M_{N^2}(\mathbb{C})$ such that the bounded anti-selfadjoint operator $\,\mathcal{I}
=1\otimes\widetilde{\mathcal{I}}$ acting on $L^2(\mathcal{A},\tau)\otimes\mathbb{C}^{\,N^2}$ satisfy
\begin{enumerate}
\item $[\,\mathcal{I},\mathcal{T}]=0$
\item $[\,\mathcal{I},\widetilde{\gamma}]=0$
\item $[\,\mathcal{I},\star]=0$
\item $[\,\mathcal{I},[\,\mathcal{I},d\,]]=-d$
\end{enumerate}
then the $N=(1,1)$ spectral data obtained in Propn. $(\ref{N=(1,1) from dynamical system})$ extends to Hermitian spectral data over $\mathcal{A}$,
i,e. $\mathcal{A}$ inherits complex structure.
\end{proposition}
\begin{proof}
We want to write $\,d:=\frac{1}{2}(\mathfrak{D}-i\overline{\mathfrak{D}})$ as $\,\partial+\overline{\partial}$ where both $\,\partial,\,\overline{\partial}$
are differentials and $\mathcal{T}=T+\overline{T}$ such that all the conditions in Defn. (\ref{N=(2,2) spectral data}) except $(6)$ are satisfied.
Consider the densely defined operator $\,d_2=[\,\mathcal{I},d\,]$ such that $[\,\mathcal{I},d_2]=-d$. This gives $\mathcal{I}^2d-2\mathcal{I}
d\mathcal{I}+d\mathcal{I}^2=-d$. Hence, $\mathcal{I}d\mathcal{I}d=\frac{1}{2}d\mathcal{I}^2d=d\mathcal{I}d\mathcal{I}$. Then,
\begin{eqnarray*}
d_2^2 & = & [\mathcal{I},d][\mathcal{I},d]\\
& = & \mathcal{I}d\mathcal{I}d-d\mathcal{I}^2d+d\mathcal{I}d\mathcal{I}\\
& = & 0
\end{eqnarray*}
i,e. $\,d_2$ is a differential. Now, define
\begin{center}
$\partial:=\frac{1}{2}(d-id_2)\quad$ and $\quad\overline{\partial}:=\frac{1}{2}(d+id_2)\,.$
\end{center}
Then, $\,d=\partial+\overline{\partial}$ and part $(1)$ in Defn. (\ref{N=(2,2) spectral data}) holds. Observe that $\{d,d_2\}=0$. Both $d$ and $d_2$
are anticommuting differentials shows that both $\partial$ and $\overline{\partial}$ are differentials. It is easy to check that $\{\partial,
\overline{\partial}\}=0$. Now, define
\begin{center}
$T:=\frac{1}{2}(\mathcal{T}-i\mathcal{I})\quad$ and $\quad\overline{T}:=\frac{1}{2}(\mathcal{T}+i\mathcal{I})\,.$
\end{center}
Then $\mathcal{T}=T+\overline{T}$ and $[T,\overline{T}]=\frac{i}{2}[\mathcal{T},\mathcal{I}]=0$. Now,
\begin{eqnarray*}
[T,\partial] & = & \frac{1}{4}([\mathcal{T},d]-i[\mathcal{T},d_2]-i[\mathcal{I},d]-[\mathcal{I},d_2])\\
& = & \frac{1}{4}(d-id_2-i[\mathcal{T}-i\mathcal{I},d_2])\\
& = & \frac{1}{2}\partial-\frac{i}{2}[T,d_2]\,.
\end{eqnarray*}
Similarly, one can show that
\begin{align*}
[\overline{T},\partial] &= \frac{1}{2}\overline{\partial}-\frac{i}{2}[\overline{T},d_2]\,,\\
[T,\overline{\partial}] &= \frac{1}{2}\partial+\frac{i}{2}[T,d_2]\,,\\
[\overline{T},\overline{\partial}] &= \frac{1}{2}\overline{\partial}+\frac{i}{2}[\overline{T},d_2]\,.\\
\end{align*}
Now, by Lemma (\ref{An operator related to d}) we know that $[\mathcal{T},d\,]=d$. Hence,
\begin{eqnarray*}
[T,d_2] & = & \frac{1}{2}(\mathcal{T}\mathcal{I}d-\mathcal{T}d\mathcal{I}-\mathcal{I}d\mathcal{T}+d\mathcal{I}\mathcal{T}-i[\mathcal{I},[\mathcal{I},d]])\\
& = & \frac{1}{2}(\mathcal{I}[\mathcal{T},d]-[\mathcal{T},d]\mathcal{I}-i[\mathcal{I},[\mathcal{I},d]]\\
& = & \frac{1}{2}(\mathcal{I}d-d\mathcal{I}-i[\mathcal{I},[\mathcal{I},d]])\\
& = & \frac{1}{2}(d_2-i[\mathcal{I},[\mathcal{I},d]])\,.
\end{eqnarray*}
Similarly, one can show that
\begin{eqnarray*}
[\overline{T},d_2] & = & \frac{1}{2}(d_2+i[\mathcal{I},[\mathcal{I},d]])\,.
\end{eqnarray*}
Hence, the following two relations
\begin{center}
$[T,d_2]=i\partial\quad$ and $\quad[\,\overline{T},d_2]=-i\overline{\partial}$
\end{center}
together is equivalent to
\begin{center}
$[\,\mathcal{I},[\,\mathcal{I},d]]=-d\,.$
\end{center}
This shows that part $(2)$ in Defn. (\ref{N=(2,2) spectral data}) holds. Both $\mathcal{I}$ and $\mathcal{T}$ commuting with $\mathcal{A}$ proves
that $[T,a]=[\overline{T},a]=0$ for all $a\in\mathcal{A}$. Now,
\begin{eqnarray*}
[d_2,a] & = & [[\,\mathcal{I},d],a]\\
& = & [\,\mathcal{I},[d,a]]
\end{eqnarray*}
Since $[d,a]$ extends to a bounded operator, we get that both $[\partial,a]$ and $[\overline{\partial},a]$ extends to bounded operators for all
$a\in\mathcal{A}$. This shows that part $(3)$ in Defn. (\ref{N=(2,2) spectral data}) holds. Now,
\begin{eqnarray*}
\{\widetilde\gamma,d_2\} & = & \widetilde\gamma[\mathcal{I},d]+[\mathcal{I},d]\widetilde\gamma\\
& = & \mathcal{I}\{\widetilde\gamma,d\}-\{\widetilde\gamma,d\}\mathcal{I}\\
& = & 0
\end{eqnarray*}
since, $\{\widetilde\gamma,d\}=0$. This shows that $\{\widetilde\gamma,\partial\}=\{\widetilde\gamma,\overline{\partial}\}=0$ i,e. part $(4)$ in
Defn. (\ref{N=(2,2) spectral data}) holds. Finally, observe that
\begin{align*}
\star\partial+\overline{\partial}^*\star &= -i(\star d_2+d_2^*\star)\\
\star\overline{\partial}+\partial^*\star &= i(\star d_2+d_2^*\star)
\end{align*}
Now, using the fact that $\mathcal{I}$ is anti-selfadjoint we see that
\begin{eqnarray*}
\star d_2+d_2^*\star & = & \star[\mathcal{I},d]+[\mathcal{I},d]^*\star\\
& = & \star\mathcal{I}d-\star d\mathcal{I}+d^*\mathcal{I}^*\star-\mathcal{I}^*d^*\star\\
& = & \mathcal{I}(\star d+d^*\star)-(\star d+d^*\star)\mathcal{I}\\
& = & 0
\end{eqnarray*}
which shows that part $(5)$ in Defn. (\ref{N=(2,2) spectral data}) holds for the phase $\zeta=-1$. Hence, existence of such suitable anti-selfadjoint
operator $\mathcal{I}$ guarantees that the $N=(1,1)$ spectral data obtained in Propn. (\ref{N=(1,1) from dynamical system}) extends to
Hermitian spectral data over $\mathcal{A}$, i,e. $\mathcal{A}$ inherits complex structure.
\end{proof}

\begin{proposition}\label{sufficient condition for Kahler structure}
The Hermitian spectral data obtained in Propn. $(\ref{sufficient condition for complex structure})$ extends to $N=(2,2)$ K\"ahler spectral
data over $\mathcal{A}$, i,e. $\mathcal{A}$ inherits K\"ahler structure, if and only if $\{d,d_2^*\}=\{d^*,d_2\}=0$ with $\,d_2=[\,\mathcal{I},d\,]$.
\end{proposition}
\begin{proof}
Recall part $(6)$ in Defn. (\ref{N=(2,2) spectral data}) which is precisely the K\"ahler condition. Observe that
\begin{eqnarray*}
\{\partial,\partial^*\} & = & \partial\partial^*+\partial^*\partial\\
& = & (d-id_2)(d^*+id_2^*)+(d^*+id_2^*)(d-id_2)\\
& = & \{d,d^*\}+\{d_2,d_2^*\}+i\{d,d_2^*\}-i\{d^*,d_2\}
\end{eqnarray*}
Similarly,
\begin{align*}
\{\overline{\partial},\overline{\partial}^*\} &= \{d,d^*\}+\{d_2,d_2^*\}-i\{d,d_2^*\}+i\{d^*,d_2\}\\
\{\partial,\overline{\partial}^*\} &= \{d,d^*\}-\{d_2,d_2^*\}-i\{d,d_2^*\}-i\{d^*,d_2\}\\
\{\overline{\partial},\partial^*\} &= \{d,d^*\}-\{d_2,d_2^*\}+i\{d,d_2^*\}+i\{d^*,d_2\}
\end{align*}
This shows that the following conditions
\begin{enumerate}
\item $\{d,d_2^*\}=\{d^*,d_2\}=0$
\item $\{d,d^*\}=\{d_2,d_2^*\}$
\end{enumerate}
are necessary and sufficient for the complex structure obtained in Propn. (\ref{sufficient condition for complex structure}) to extend to
K\"ahler structure on $\mathcal{A}$. However, condition $(2)$ follows from condition $(1)$ because
\begin{eqnarray*}
\{d,d^*\} & = & dd^*+d^*d\\
& = & -[\,\mathcal{I},[\,\mathcal{I},d\,]]d^*-d^*[\,\mathcal{I},[\,\mathcal{I},d\,]]\\
& = & -\mathcal{I}d_2d^*+d_2\mathcal{I}d^*-d^*\mathcal{I}d_2+d^*d_2\mathcal{I}\\
& = & (d_2\mathcal{I}d^*-d_2d^*\mathcal{I})+d_2d^*\mathcal{I}+(\mathcal{I}d^*d_2-d^*\mathcal{I}d_2)-\mathcal{I}d^*d_2+d^*d_2\mathcal{I}-\mathcal{I}d_2d^*\\
& = & (d_2d_2^*+d_2^*d_2)+\{d_2,d^*\}\mathcal{I}-\mathcal{I}\{d^*,d_2\}\\
& = & \{d_2,d_2^*\}
\end{eqnarray*}
if $\{d^*,d_2\}=0$. Hence, the condition $\{d,d_2^*\}=\{d^*,d_2\}=0$, with $\,d_2=[\,\mathcal{I},d\,]$, is necessary and sufficient for the complex
structure obtained in Propn. (\ref{sufficient condition for complex structure}) to extend to K\"ahler structure on $\mathcal{A}$.
\end{proof}

\begin{theorem}\label{N=(2,2) from dynamical system}
Let $G$ be an even dimensional, connected, abelian Lie group and $(\mathcal{A},G,\alpha,\tau)$ be a $C^*$-dynamical system equipped with a faithful
$G$-invariant trace $\tau$. If it determines a $N=1$ spectral data $(\mathcal{A},\mathcal{H},D,\sigma)$ then it always extends to $N=(2,2)$
K\"ahler spectral data over $\mathcal{A}$, i,e. $\mathcal{A}$ inherits a K\"ahler structure. 
\end{theorem}
\begin{proof}
Let $dim(G)=2k$ and $N=2^k$. We first produce a skew-adjoint matrix $\,\widetilde{\mathcal{I}}\in M_{N^2}(\mathbb{C})$ such that the anti-selfadjoint
operator $\mathcal{I}=1\otimes\widetilde{\mathcal{I}}$ acting on $L^2(\mathcal{A},\tau)\otimes\mathbb{C}^{N^2}$ satisfy all the conditions of Propn.
(\ref{sufficient condition for complex structure}). Note that $M_{N^2}(\mathbb{C})=M_N(\mathbb{C})\otimes_\mathbb{C}M_N(\mathbb{C})$. Consider the
Clifford algebra $\mathbb{C}l(2k)$ and suppose that $\{e_1,\ldots,e_{2k}\}$ be the generating set. Consider the following elements
\begin{eqnarray}\label{required skew-adjoint matrices}
A(\ell,j) & := & 1\otimes e_\ell e_j+e_\ell e_j\otimes 1
\end{eqnarray}
in $\mathbb{C}l(2k)\otimes\mathbb{C}l(2k)$ for each pair $(\ell,j)$ with $\ell<j$ and $\ell,j\in\{1,\ldots,2k\}$. We claim that each $A(\ell,j)$
commutes with the elements $e_1\ldots e_{2k}\otimes e_1\ldots e_{2k}$ and $1\otimes e_1\ldots e_{2k}$ of $\mathbb{C}l(2k)\otimes\mathbb{C}l(2k)$.
This is because
\begin{eqnarray*}
e_\ell e_j(e_1\ldots e_{2k}) & = & (-1)^{j-1}(-1)e_1\ldots e_\ell\ldots\widehat{e_j}\ldots e_{2k}\\
& = & (-1)^{j-1}(-1)(-1)^{\ell-1}(-1)e_1\ldots\widehat{e_\ell}\ldots\widehat{e_j}\ldots e_{2k}\\
& = & (-1)^{\ell+j}e_1\ldots\widehat{e_\ell}\ldots\widehat{e_j}\ldots e_{2k}
\end{eqnarray*}
and
\begin{eqnarray*}
(e_1\ldots e_{2k})e_\ell e_j & = & (-1)^{2n-\ell}(-1)(e_1\ldots\widehat{e_\ell}\ldots e_j\ldots e_{2k})e_j\\
& = & (-1)^{2n-\ell+1}(-1)^{2n-j}(-1)e_1\ldots\widehat{e_\ell}\ldots\widehat{e_j}\ldots e_{2k}\\
& = & (-1)^{-\ell-j}e_1\ldots\widehat{e_\ell}\ldots\widehat{e_j}\ldots e_{2k}
\end{eqnarray*}
where, $\,\widehat{}\,$ means the corresponding term is omitted. That is, we are getting $e_\ell e_j$ commutes with $e_1\ldots e_{2k}\in\mathbb{C}l(2k)$.
Now, it is easy to verify that for each such pair $(\ell,j)$, the element $A(\ell,j)$ commutes with $\sum_{r\neq \ell,j}e_r\otimes e_r$ in
$\mathbb{C}l(2k)\otimes\mathbb{C}l(2k)$. Observe that $A(\ell,j)$ also commutes with $e_\ell\otimes e_\ell+e_j\otimes e_j$. Hence, for each such
pair $(\ell,j),\,A(\ell,j)$ will commute with $\sum_{r=1}^{2k}e_r\otimes e_r$ in $\mathbb{C}l(2k)\otimes\mathbb{C}l(2k)$. Now, let $\pi:\mathbb{C}l
(2k)\longrightarrow M_N(\mathbb{C})$, given by $\pi:e_r\longmapsto\gamma_r$, be the irreducible representation in Propn. (\ref{Clifford representation}).
The element $\prod_{j=1}^{2k}e_j\in\mathbb{C}l(2k)$ corresponds to the grading operator $\sigma$ if $k$ is even and $-i\sigma$ if $k$ is odd under
the representation $\pi$ (see \cite{DD} for detail). Hence,
\begin{center}
$(\pi\otimes\pi)(A(\ell,j))=1\otimes\gamma_\ell\gamma_j+\gamma_\ell\gamma_j\otimes 1$
\end{center}
are skew-adjoint matrices in $M_N(\mathbb{C})\otimes M_N(\mathbb{C})$ such that the anti-selfadjoint operators $1\otimes\mathcal{I}_{(\ell,j)}:=1
\otimes(\pi\otimes\pi)(A(\ell,j))$ commute with $\mathcal{T},\widetilde{\gamma}$ and $\star$ (recall the expression of $\mathcal{T}$ from Lemma
[\ref{An operator related to d}] and that of $\,\widetilde{\gamma},\star$ from Propn. [\ref{N=(1,1) from dynamical system}]). Observe that
\begin{eqnarray*}
[1\otimes\mathcal{I}_{(\ell,j)},d\,] & = & \sum_{r=1}^{2k}\frac{1}{2}\partial_r\otimes 1\otimes[\gamma_\ell\gamma_j,\gamma_r]+\frac{i\varepsilon^\prime}{2}
\partial_r\otimes[\gamma_\ell\gamma_j,\gamma_r]\otimes\sigma\\
& = & \partial_\ell\otimes 1\otimes\gamma_j-\partial_j\otimes 1\otimes\gamma_\ell+i\varepsilon^\prime(\partial_\ell\otimes\gamma_j\otimes\sigma-
\partial_j\otimes\gamma_\ell\otimes\sigma)
\end{eqnarray*}
and hence,
\begin{center}
$[1\otimes\mathcal{I}_{(\ell,j)},[1\otimes\mathcal{I}_{(\ell,j)},d\,]\,]=-2(\partial_\ell\otimes 1\otimes\gamma_\ell+\partial_j\otimes 1\otimes
\gamma_j)-2i\varepsilon^\prime(\partial_\ell\otimes\gamma_\ell\otimes\sigma+\partial_j\otimes\gamma_j\otimes\sigma)\,.$
\end{center}
Hence, if we consider $\,\mathcal{I}=1\otimes\widetilde{\mathcal{I}}$ with
\begin{center}
$\widetilde{\mathcal{I}}=\frac{1}{2}\left(\mathcal{I}_{(1,2)}+\mathcal{I}_{(3,4)}+\ldots+\mathcal{I}_{(2k-1,2k)}\right)$
\end{center}
then we have $[\,\mathcal{I},[\,\mathcal{I},d\,]\,]=-d$ along with $[\,\mathcal{I},\mathcal{T}]=[\,\mathcal{I},\widetilde\gamma\,]=[\,\mathcal{I},
\star]=0$. Hence, by Propn. (\ref{sufficient condition for complex structure}), the $N=1$ spectral data $(\mathcal{A},\mathcal{H},D,\sigma)$
extends to Hermitian spectral data over $\mathcal{A}$, i,e. $\mathcal{A}$ inherits complex structure.

We now show that the condition in Propn. (\ref{sufficient condition for Kahler structure}) is also satisfied. For $d_2:=[\,\mathcal{I},d\,]$,
note that
\begin{eqnarray*}
d_2=\sum_{j=1,\,j\,odd}^{2k}\frac{1}{2}(\partial_j\otimes 1\otimes\gamma_{j+1}-\partial_{j+1}\otimes 1\otimes\gamma_j)+\frac{i\varepsilon^\prime}{2}
(\partial_j\otimes\gamma_{j+1}\otimes\sigma-\partial_{j+1}\otimes\gamma_j\otimes\sigma)
\end{eqnarray*}
and recall that $d^*=\sum_{\ell=1}^{2k}\frac{1}{2}(\partial_\ell\otimes 1\otimes\gamma_\ell-i\varepsilon^\prime\partial_\ell\otimes\gamma_\ell
\otimes\sigma)$. Then,
\begin{eqnarray*}
&  & 4\{d^*,d_2\}\\
& = & \sum_{j=1,\,j\,odd}^{2k}\,\sum_{\ell=1}^{2k}(\partial_\ell\partial_j\otimes 1\otimes\{\gamma_\ell,\gamma_{j+1}\}-\partial_\ell\partial_{j+1}\otimes 1
\otimes\{\gamma_\ell,\gamma_j\}+\partial_\ell\partial_j\otimes\{\gamma_\ell,\gamma_{j+1}\}\otimes 1\\
&  & \quad\quad\quad\quad\quad-\partial_\ell\partial_{j+1}\otimes\{\gamma_\ell,\gamma_j\}\otimes 1)\\
& = & -4\sum_{j=1,\,j\,odd}^{2k}\partial_{j+1}\partial_j\otimes 1\otimes 1+4\sum_{j=1,\,j\,odd}^{2k}\partial_j\partial_{j+1}\otimes 1\otimes 1\\
& = & 0
\end{eqnarray*}
since, $\mathfrak{g}$ is abelian. Similarly, one can verify that $\{d,d_2^*\}=0$. Hence, the $N=1$ spectral data $(\mathcal{A},\mathcal{H},D,\sigma)$
extends to $N=(2,2)$ K\"ahler spectral data over $\mathcal{A}$, i,e. $\mathcal{A}$ inherits K\"ahler structure.
\end{proof}

\begin{proposition}\label{total number of Kahler structure}
There can be obtained $\,\prod_{j=1,\,j\,odd}^{dim(G)-1}(dim(G)-j)$ different K\"ahler structures in previous Thm. $(\ref{N=(2,2) from dynamical system})$.
\end{proposition}
\begin{proof}
Let $dim(G)=2k$. In the previous Thm. (\ref{N=(2,2) from dynamical system}), we produced the differential $d_2=[\,\mathcal{I},d\,]$ by taking a
particular $\mathcal{I}=1\otimes\widetilde{\mathcal{I}}$ where $\widetilde{\mathcal{I}}=\frac{1}{2}\left(\mathcal{I}_{(1,2)}+\mathcal{I}_{(3,4)}+
\ldots+\mathcal{I}_{(2k-1,2k)}\right)$. We now show that there are $\,\prod_{j=1,\,j\,odd}^{2k-1}(2k-j)$ different choice for $\widetilde{\mathcal{I}}$
built out of $\mathcal{I}_{(\ell,j)}$ with $\ell<j$ and $\ell,j\in\{1,\ldots,2k\}$. First choose $\mathcal{I}_{(1,j)}$ with $j>1$.
Total number of choice is $2k-1$.\\
\underline{Case 1:}~ If $j=2$, next choose $\mathcal{I}_{(3,r)}$ with $r>3$.\\
\underline{Case 2:}~ If $j>2$, next choose $\mathcal{I}_{(2,r)}$ so that $r>2$ and $r\in\{1,2,\ldots,2k\}\smallsetminus\{1,2,j\}$.\\
Hence for each $\mathcal{I}_{(1,j)}$, we get a total $2k-3$ different choice to consider the next $\mathcal{I}_{(3,r)}$ or $\mathcal{I}_{(2,r)}$
accordingly as $j=2$ or $j>2$ respectively. Now,\\
\underline{Case 1:}~ If $j=2$ and $\mathcal{I}_{(3,r)}$ with $r>3$ have been chosen, next consider $\mathcal{I}_{(s,t)}$ with $s=min\{\{1,2,\ldots,
2k\}\smallsetminus\{1,2,3,r\}\}$ and $t>s$ with $t\in\{1,2,\ldots,2k\}\smallsetminus\{1,2,3,r\}$.\\
\underline{Case 2:}~ If $j>2$ and $\mathcal{I}_{(2,r)}$ with $r>2$ have been chosen, next consider $\mathcal{I}_{(s,t)}$ with $s=min\{\{1,2,\ldots,
2k\}\smallsetminus\{1,2,j,r\}\}$ and $t>s$ with $t\in\{1,2,\ldots,2k\}\smallsetminus\{1,2,j,r\}$.\\
Hence for each $\mathcal{I}_{(3,r)}$, we get a total $2k-5$ different choice to choose the next $\mathcal{I}_{(s,t)}$, and similar choice
for each $\mathcal{I}_{(2,r)}$.

Until this we get a total $(2k-1)(2k-3)(2k-5)$ different choice. Proceed like this to choose the next $\mathcal{I}_{(p,q)}$ with $\mathcal{I}_{(s,t)}$
being chosen. We will finally get total $\prod_{j=1,\,j\,odd}^{2k-1}(2k-j)$ different choice of $\widetilde{\mathcal{I}}$. It is a purely algebraic
verification that all these choice of $\mathcal{I}=1\otimes\widetilde{\mathcal{I}}$ give us different $d_2$ satisfying $[\,\mathcal{I},[\,\mathcal{I},
d]\,]=-d$ and $\{d^*,d_2\}=\{d,d_2^*\}=0$. Thus, one can obtain $\prod_{j=1,\,j\,odd}^{2k-1}(2k-j)$ different K\"ahler structures in previous Thm.
$(\ref{N=(2,2) from dynamical system})$.

These various choice of indices $(m,n)$ in $\mathcal{I}_{(m,n)}$ at each stage is best understood by a directed tree. For example, if $dim(G)=2k=2$
then there is a unique choice of $\widetilde{\mathcal{I}}$ namely, $\widetilde{\mathcal{I}}=\frac{1}{2}\mathcal{I}_{(1,2)}$. If $dim(G)=2k=4$ then
we have the following tree for various choice of $\mathcal{I}_{(\ell,j)}$ at each stage,
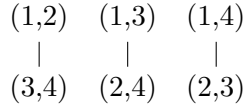
\begin{center}
$\Tree[.(1,2) [.(3,4) ]]\quad
\Tree[.(1,3) [.(2,4) ]]\quad
\Tree[.(1,4) [.(2,3) ]]$
\end{center}
Here, the top index represents different possible choice of $\mathcal{I}_{(1,j)}$ and we get total three different choice of $\widetilde{\mathcal{I}}$
namely, $\widetilde{\mathcal{I}}=\frac{1}{2}(\mathcal{I}_{(1,2)}+\mathcal{I}_{(3,4)}),\,\widetilde{\mathcal{I}}=\frac{1}{2}(\mathcal{I}_{(1,3)}+
\mathcal{I}_{(2,4)})$ and $\widetilde{\mathcal{I}}=\frac{1}{2}(\mathcal{I}_{(1,4)}+\mathcal{I}_{(2,3)})$. If $dim(G)=2k=6$ then we have the following
tree for various choice of $\mathcal{I}_{(i,j)}$ at each stage,
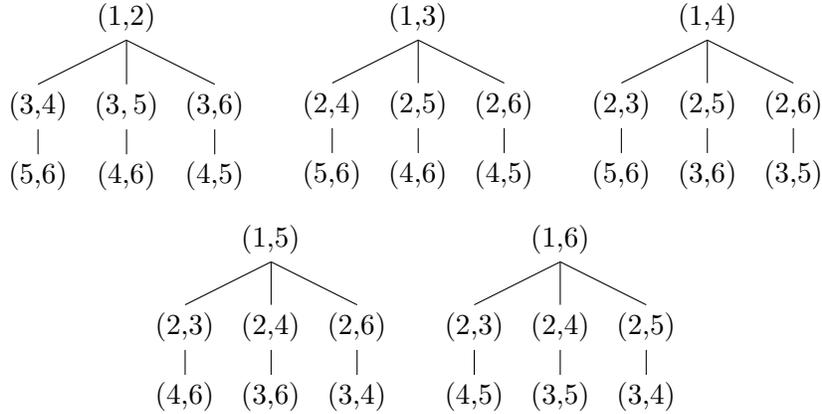
\begin{center}
$\Tree[.(1,2) [.(3,4) [.(5,6) ]]
[.(3,\,5) [.(4,6) ]]
[.(3,6) [.(4,5) ]]]\quad\quad
\Tree[.(1,3) [.(2,4) [.(5,6) ]]
[.(2,5) [.(4,6) ]]
[.(2,6) [.(4,5) ]]]\quad\quad\Tree[.(1,4) [.(2,3) [.(5,6) ]]
[.(2,5) [.(3,6) ]]
[.(2,6) [.(3,5) ]]]$
\end{center}
\begin{center}
$\Tree[.(1,5) [.(2,3) [.(4,6) ]]
[.(2,4) [.(3,6) ]]
[.(2,6) [.(3,4) ]]]\quad\quad\Tree[.(1,6) [.(2,3) [.(4,5) ]]
[.(2,4) [.(3,5) ]]
[.(2,5) [.(3,4) ]]]$
\end{center}
The top index represents different possible choice of $\mathcal{I}_{(1,j)}$ and we get total fifteen different choice of $\widetilde{\mathcal{I}}$
given by half times the addition of each vertical row along their prescribed path. Observe that the $\widetilde{\mathcal{I}}$ given by half times
the addition of the first vertical row namely, $\widetilde{\mathcal{I}}=\frac{1}{2}(\mathcal{I}_{(1,2)}+\mathcal{I}_{(3,4)}+\mathcal{I}_{(5,6)})$
is the one considered in previous Thm. (\ref{N=(2,2) from dynamical system}).
\end{proof}

Now we explain why we did not discard $\,\varepsilon^\prime$ in every places from Propn. (\ref{self-adjoint operator}) up to Thm.
(\ref{N=(2,2) from dynamical system}) (see Remark [\ref{why to keep epsilon}]). Reason is that as pointed out in (\cite{DD}), in the even case there
are actually two possible real structures $J_\pm$ that differ by multiplication by the grading operator. None of them should be preferred as they
are perfectly on the same footing. The table mentioned in Propn. (\ref{Clifford representation}) has the following extension :
\medskip

\begin{center}
\begin{tabular}{|c| c c c c| c c c c| c c c c| }
\hline
$n$ & $0$ & $2$ & $4$ & $6$ & $0$ & $2$ & $4$ & $6$ & $1$ & $3$ & $5$ & $7$ \\ [0.8ex]
\hline
$\varepsilon$ & $+$ & $-$ & $-$ & $+$ & $+$ & $+$ & $-$ & $-$ & $+$ & $-$ & $-$ & $+$\\
$\varepsilon^\prime$  & $+$ & $+$ & $+$ & $+$ & $-$ & $-$ & $-$ & $-$ & $-$ & $+$ & $-$ & $+$\\
$\varepsilon^{\prime\prime}$ & $+$ & $-$ & $+$ & $-$ & $+$ & $-$ & $+$ & $-$ & $\,$ & $\,$ & $\,$ & $\,$\\
\hline
\end{tabular}
\end{center}
\medskip

The first column represents the real structure $J_+$ and the second is for $J_-$. To accommodate both the possible real structures we did not discard
$\,\varepsilon^\prime$. Hence, accordingly as $\,\varepsilon^\prime=+1$ or $-1$, both $\,\partial$ and $\,\overline{\partial}$ changes and we actually
obtain two different K\"ahler structures in Thm. (\ref{N=(2,2) from dynamical system}), and therefore $2\prod_{j=1,\,j\,odd}^{dim(G)-1}(dim(G)-j)$
different K\"ahler structures in view of Propn. (\ref{total number of Kahler structure}). However, it turns out that these two set of K\"ahler
differentials corresponding to $J_\pm$ are unitary conjugate to each other. If we denote the K\"ahler differentials obtained in Th.
(\ref{N=(2,2) from dynamical system}) by $\partial_\pm$ and $\overline{\partial_\pm}$ corresponding to the real structures $J_\pm\,$, then one can
verify the following relationship
\begin{center}
$(1\otimes\sigma\otimes 1)\,\partial_+=\partial_-\,(1\otimes\sigma\otimes 1)\quad and\quad(1\otimes\sigma\otimes 1)\,\overline{\partial_{+}}=
\overline{\partial_{-}}\,(1\otimes\sigma\otimes 1)\,.$
\end{center}
Here, the operator $1\otimes\sigma\otimes 1$, which is a self-adjoint unitary acting on $L^2(\mathcal{A},\tau)\otimes\mathbb{C}^N\otimes\mathbb{C}^N$,
is precisely the product of the $\mathbb{Z}_2$-grading and the Hodge operator obtained in Proposition (\ref{N=(1,1) from dynamical system}). Being
unitary equivalent we do not distinguish between the K\"ahler structures $\{\partial_+,\,\overline{\partial_+}\}$ and $\{\partial_-,\,\overline{\partial_-}\}$.
With this fact, combining Thm. (\ref{N=(2,2) from dynamical system}) and Propn. (\ref{total number of Kahler structure}) we conclude the following
final theorem.

\begin{theorem}\label{final theorem}
Let $G$ be an even dimensional, connected, abelian Lie group and $(\mathcal{A},G,\alpha,\tau)$ be a $C^*$-dynamical system equipped with a faithful
$G$-invariant trace $\tau$. If it determines a $\varTheta$-summable even spectral triple, then $\mathcal{A}$ inherits at least $\,\prod_{j=1,
\,j\,odd}^{\,dim(G)}(dim(G)-j)$ different K\"ahler structures.
\end{theorem}

\begin{remark}\rm
This theorem tells us that in these cases of $C^*$-dynamical systems existence of K\"ahler structure depends only on the fact that whether the
densely defined unbounded symmetric operator $D:=\sum_{j=1}^{2k}\partial_j\otimes\gamma_j$ extends to a self-adjoint operator with compact resolvent.
\end{remark}

\begin{corollary}\label{compact ergodic action}
If $(\mathcal{A},\mathbb{T}^{2k},\alpha)$ is a $C^*$-dynamical system such that the action of $\,\mathbb{T}^{2k}$ is ergodic, then $\mathcal{A}$
inherits at least $\,\prod_{j=1,\,j\,odd}^{\,2k}(2k-j)$ different K\"ahler structures.
\end{corollary}
\begin{proof}
Since $\mathbb{T}^{2k}$ is compact and the action is ergodic, the unique $\mathbb{T}^{2k}$-invariant state becomes a faithful trace (\cite{HLS}),
and we have a $\,2k$-summable (and hence $\varTheta$-summable \cite{Con7}) even spectral triple (Thm. $[5.4]$ in \cite{Gab}). Conclusion now
follows from Thm. (\ref{final theorem}).
\end{proof}

\begin{corollary}
For $n$ even, the noncommutative $n$-torus $\mathcal{A}_\Theta$ satisfies the $N=(2,2)$ K\"ahler spectral data, i,e. they are noncommutative
K\"ahler manifolds. Moreover, there are at least $\prod_{j=1,\,j\,odd}^{\,n}(n-j)$ different K\"ahler structures on $\mathcal{A}_\Theta$.
\end{corollary}
\begin{proof}
It is well known that the $C^*$-dynamical system $(\mathcal{A}_\Theta,\mathbb{T}^n,\alpha)$ on the noncommutative $n$-torus $\mathcal{A}_\Theta$,
where $\,\alpha_{\bf z}(U_k):=z_kU_k,\,k=1,\ldots,n$, equipped with a unique $\,\mathbb{T}^n$-invariant faithful trace given by
\begin{center}
$\tau\left(\sum\alpha_{(m_1,\ldots,m_n)}U_1^{m_1}\ldots U_n^{m_n}\right):=\alpha_{\bf 0}$
\end{center}
with $\,\alpha_{(m_1,\ldots,m_n)}\in\mathbb{S}(\mathbb{Z}^n)$, gives a $n$-summable (and hence $\varTheta$-summable \cite{Con7}) spectral triple
\begin{center}
$\left(\mathcal{A}_\Theta\,,\,\ell^2(\mathbb{Z}^n)\otimes\mathbb{C}^{2^{\lfloor n/2\rfloor}},\,D:=\sum_{j=1}^n\,\partial_j\otimes\gamma_j\right)\,.$
\end{center}
This spectral triple is even if $n$ is even, and we obtain a $N=1$ spectral data on $\mathcal{A}_\Theta$. Conclusion now follows from Thm.
(\ref{final theorem}).
\end{proof}

\begin{remark}\rm
As mentioned earlier in the Introduction, characterizing holomorphic structures on $n$-dimensional manifolds, with $n>2$, via positive Hochschild
cocycles is still open. That is why methods in (\cite{Con3}) does not extend to noncommutative higher dimensional tori.
\end{remark}

\begin{corollary}\label{the case of noncommutative torus}
For the noncommutative two-torus $\mathcal{A}_\theta$, with irrational $\,\theta$, represented faithfully on the Hilbert space $\,\ell^2(\mathbb{Z}^2)
\otimes\mathbb{C}^2\bigoplus\ell^2(\mathbb{Z}^2)\otimes\mathbb{C}^2\cong\ell^2(\mathbb{Z}^2)\otimes\mathbb{C}^4$ by diagonal operator,
\begin{center}
$\gamma_1\,=\,i\begin{pmatrix}
0 & 1\\
1 & 0
\end{pmatrix}\quad,\quad\gamma_2\,=\,i\begin{pmatrix}
0 & -i\\
i & 0
\end{pmatrix}\quad,\quad\sigma\,=\,\begin{pmatrix}
1 & 0\\
0 & -1
\end{pmatrix}\,.$ 
\end{center}
The Dirac operator is given by
\begin{center}
$D\,=\,\begin{pmatrix}
             0 & i\partial_1+\partial_2 & 0 & 0\\
             i\partial_1-\partial_2 & 0 & 0 & 0\\
             0 & 0 & 0 & i\partial_1+\partial_2\\
             0 & 0 & i\partial_1-\partial_2 & 0 
            \end{pmatrix}$ 
\end{center}
with the grading operator $\widetilde{\gamma}$ and the Hodge operator $\star$ as
\begin{center}
$\widetilde{\gamma}=\begin{pmatrix}
1 & 0 & 0 & 0\\
0 & -1 & 0 & 0\\
0 & 0 & -1 & 0\\
0 & 0 & 0 & 1
\end{pmatrix}\quad,\quad\star=\begin{pmatrix}
1 & 0 & 0 & 0\\
0 & -1 & 0 & 0\\
0 & 0 & 1 & 0\\
0 & 0 & 0 & -1
\end{pmatrix}\,$.
\end{center}
The two set of unitary equivalent K\"ahler differentials are given by
\begin{align*}
\partial &\,= \begin{pmatrix}
0 & 0 & 0 & 0\\
\frac{1}{2}(i\partial_1-\partial_2) & 0 & 0 & 0\\
\frac{i\varepsilon^\prime}{2}(i\partial_1-\partial_2) & 0 & 0 & 0\\
0 & -\frac{i\varepsilon^\prime}{2}(i\partial_1-\partial_2) & \frac{1}{2}(i\partial_1-\partial_2) & 0\\
\end{pmatrix}\\
\overline{\partial} &\,= \begin{pmatrix}
0 & \frac{1}{2}(i\partial_1+\partial_2) & \frac{i\varepsilon^\prime}{2}(i\partial_1+\partial_2) & 0\\
0 & 0 & 0 & -\frac{i\varepsilon^\prime}{2}(i\partial_1+\partial_2)\\
0 & 0 & 0 & \frac{1}{2}(i\partial_1+\partial_2)\\
0 & 0 & 0 & 0 
\end{pmatrix}
\end{align*}
with $\,\varepsilon^\prime=\pm 1$, and the nilpotent differential is $\,d:=\partial+\overline{\partial}$ with $d+d^*=D$.
\end{corollary}
\begin{proof}
The matrices $\gamma_1,\gamma_2$ and $\sigma$ are obtained from the explicit representation of $\mathbb{C}l(2)$ on $M_2(\mathbb{C})$ (see Propn.
[\ref{Clifford representation}]). From Propn. (\ref{N=(1,1) from dynamical system}), since $d=\frac{1}{2}(\mathfrak{D}-i\overline{\mathfrak{D}}\,)$,
we get
\begin{center}
$d=\frac{1}{2}(\partial_1\otimes 1\otimes\gamma_1+\partial_2\otimes 1\otimes\gamma_2)+\frac{i\varepsilon^\prime}{2}(\partial_1\otimes\gamma_1\otimes
\sigma+\partial_2\otimes\gamma_2\otimes\sigma)$
\end{center}
and from Thm. (\ref{N=(2,2) from dynamical system}), we get
\begin{center}
$d_2=\frac{1}{2}(\partial_1\otimes 1\otimes\gamma_2-\partial_2\otimes 1\otimes\gamma_1)+\frac{i\varepsilon^\prime}{2}(\partial_1\otimes\gamma_2
\otimes\sigma-\partial_2\otimes\gamma_1\otimes\sigma)$
\end{center}
The expression for the Dirac operator $D$ is then clear since, $D=d+d^*=\mathfrak{D}$. The two set of unitary equivalent K\"ahler differentials are
given by $\partial=\frac{1}{2}(d-id_2)$ and $\overline{\partial}=\frac{1}{2}(d+id_2)$ with $\,\varepsilon^\prime=\pm 1$.
\end{proof}

\begin{remark}\rm
The differential $\,\overline{\partial}\,$ in the above Cor. (\ref{the case of noncommutative torus}) coincides with the complex structure obtained
in (\cite{Con3}) from cyclic cohomology and using the equivalence of conformal and complex structures in two dimensions. This is further considered
in (\cite{PS}). We will come back to it towards the end of next section.
\end{remark}
\bigskip


\section{Category of Holomorphic vector bundle}

\subsection{Holomorphic vector bundle}\texorpdfstring{$\newline$}{Lg}
Let $(\mathcal{A},\mathcal{H},\partial,\overline{\partial},T,\overline{T},\gamma,\star)$ be a Hermitian (or in particular, $N=(2,2)$ K\"ahler)
spectral data over the unital algebra $\mathcal{A}$. Recall the space of complex differential forms from Section $[2.3.2]$ and notion of integration
from Section $[2.3.3]$ in (\cite{FGR2}). A crucial orthogonality property is mentioned in Propn. $[2.35]$ in (\cite{FGR2}). However, for this
subsection it is enough to recall that
\begin{align*}
\Omega^{1,0}_{\partial,\overline{\partial}}(\mathcal{A}):=span\{a[\partial,b]:a,b\in\mathcal{A}\}\quad &, \quad
\Omega^{2,0}_{\partial,\overline{\partial}}(\mathcal{A}):=span\{a[\partial,b][\partial,c]:a,b,c\in\mathcal{A}\}\,\,,\\
\Omega^{0,1}_{\partial,\overline{\partial}}(\mathcal{A}):=span\{a[\,\overline{\partial},b]:a,b\in\mathcal{A}\}\quad &, \quad
\Omega^{0,2}_{\partial,\overline{\partial}}(\mathcal{A}):=span\{a[\,\overline{\partial},b][\,\overline{\partial},c]:a,b,c\in\mathcal{A}\}\,\,.
\end{align*}

\begin{definition}[\cite{KLS}]\label{algebra of holomorphic elements}
The algebra of holomorphic elements in $\mathcal{A}$ is defined as
\begin{center}
$\mathcal{O}(\mathcal{A}):=Ker\left\{\,\overline{\partial}:\mathcal{A}\longrightarrow\Omega^{0,1}_{\partial,\overline{\partial}}(\mathcal{A})\right\}\,.$
\end{center}
This is a $\mathbb{C}$-subalgebra of $\mathcal{A}$.
\end{definition}

\begin{definition}[\cite{KLS}]\label{holomorphic vector bundle}
A holomorphic structure on a f.g.p. left $\mathcal{A}$-module $\mathcal{E}$ is a flat $\overline{\partial}$-connection, i,e. connection $\nabla:
\mathcal{E}\longrightarrow\Omega^{0,1}_{\partial,\overline{\partial}}(\mathcal{A})\otimes_\mathcal{A}\mathcal{E}$ such that the associated
$\overline{\partial}$-curvature $\varTheta\in\mathcal{H}om_\mathcal{A}\left(\mathcal{E}\,,\Omega^{0,2}_{\partial,\overline{\partial}}(\mathcal{A})
\otimes_\mathcal{A}\mathcal{E}\right)$ vanishes. The pair $(\mathcal{E},\nabla)$ is called a holomorphic vector bundle over $\mathcal{A}$.
\end{definition}

\begin{definition}[\cite{KLS}]\label{holomorphic section}
If $(\mathcal{E},\nabla)$ is a holomorphic vector bundle over $\mathcal{A}$ then
\begin{center}
$H^0(\mathcal{E},\nabla):=ker\left\{\nabla:\mathcal{E}\longrightarrow\Omega^{0,1}_{\partial,\overline{\partial}}(\mathcal{A})\otimes_\mathcal{A}
\mathcal{E}\right\}$
\end{center}
is called the space of holomorphic sections on $\,\mathcal{E}$.
\end{definition}

\begin{definition}[\cite{BS}]
The holomorphic vector bundles are the objects in a category $\mathcal{H}o\ell(\mathcal{A})$. A morphism $\Phi:(\mathcal{E}_1,\nabla_1)\longrightarrow
(\mathcal{E}_2,\nabla_2)$ is a left $\mathcal{A}$-module map $\phi:\mathcal{E}_1\rightarrow\mathcal{E}_2$ satisfying $\nabla_2\circ\phi=(id\otimes
\phi)\circ\nabla_1$.
\end{definition}

\begin{remark}\rm
\begin{enumerate}
\item It follows from the definition of connection that $\,H^0(\mathcal{E},\nabla)$ is a left $\mathcal{O}(\mathcal{A})$-module.
\item Recall that in the classical case, a vector bundle on a complex manifold is holomorphic if and only if it admits a flat $\overline{\partial}$-connection.
\end{enumerate}
\end{remark}

Consider a f.g.p. left module $\mathcal{E}$ over $\mathcal{A}$. Then there exists a positive integer $m$ and a left $\mathcal{A}$-module homomorphism
$pr:\mathcal{A}^m\longrightarrow\mathcal{E}$. By definition, there exists a left $\mathcal{A}$-module $\mathcal{F}$ such that $\mathcal{E}\oplus
\mathcal{F}\cong\mathcal{A}^m$ and denote $i:\mathcal{E}\longrightarrow\mathcal{A}^m$ to be the inclusion map determined by this isomorphism. We
have $pr\circ i=id$ on $\mathcal{E}$.

\begin{lemma}\label{to be used in next lemma}
Any $\overline{\partial}$-connection $\widetilde\nabla$ on the free module $\mathcal{A}^m$ induces a $\overline{\partial}$-connection $\nabla$ on
$\,\mathcal{E}$.
\end{lemma}
\begin{proof}
Given such $\widetilde\nabla$, define
\begin{align*}
\nabla:\mathcal{E} &\longrightarrow \Omega_{\partial,\overline\partial}^{0,1}(\mathcal{A})\otimes_\mathcal{A}\mathcal{E}\\
\nabla &= (id\otimes pr)\circ\widetilde{\nabla}\circ i\,.
\end{align*}
Clearly, $\nabla$ is a $\mathbb{C}$-linear map. Now, for all $a\in\mathcal{A}$ and $\xi\in\mathcal{E}$,
\begin{eqnarray*}
\nabla(a\xi) & = & (id\otimes pr)\circ\widetilde{\nabla}(ai(\xi))\\
& = & (id\otimes pr)\left([\,\overline\partial,a]\otimes i(\xi)+a\widetilde{\nabla}\circ i(\xi)\right)\\
& = & [\,\overline\partial,a]\otimes\xi+a\nabla(\xi)
\end{eqnarray*}
proving $\nabla$ is a $\overline\partial$-connection on $\mathcal{E}$.
\end{proof}

Moreover, the converse is also true.

\begin{proposition}
Any $\overline{\partial}$-connection $\nabla$ on $\mathcal{E}$ is induced by a $\overline{\partial}$-connection $\widetilde\nabla$ on the free
module $\mathcal{A}^m\,$.
\end{proposition}
\begin{proof}
Start with a $\overline{\partial}$-connection $\widetilde\nabla$ on the free module $\mathcal{A}^m\,$. By previous Lemma (\ref{to be used in next lemma}),
we get a $\overline{\partial}$-connection $\nabla$ on $\mathcal{E}$ by the formula $\nabla=(id\otimes pr)\circ\widetilde{\nabla}\circ i$. Now, let
$\nabla^\prime$ be any other $\overline{\partial}$-connection on $\mathcal{E}$. Then $\nabla^\prime-\nabla\in\mathcal{H}om_\mathcal{A}(\mathcal{E},
\Omega_{\partial,\overline\partial}^{0,1}(\mathcal{A})\otimes_\mathcal{A}\mathcal{E})$. Since,
\begin{center}
$id\otimes pr:\Omega_{\partial,\overline\partial}^{0,1}(\mathcal{A})\otimes_\mathcal{A}\mathcal{A}^m\longrightarrow\Omega_{\partial,\overline\partial}^{0,1}
(\mathcal{A})\otimes_\mathcal{A}\mathcal{E}$
\end{center}
is surjective and $\mathcal{E}$ is a projective module, there exists a module map
\begin{center}
$\phi:\mathcal{E}\longrightarrow\Omega_{\partial,\overline\partial}^{0,1}(\mathcal{A})\otimes_\mathcal{A}\mathcal{A}^m$
\end{center}
such that $\nabla^\prime-\nabla=(id\otimes pr)\circ\phi$. Then, $\,\widetilde\phi=\phi\circ pr\in\mathcal{H}om_\mathcal{A}(\mathcal{A}^m\,,\,
\Omega_{\partial,\overline\partial}^{0,1}(\mathcal{A})\otimes_\mathcal{A}\mathcal{A}^m)$ and hence, $\widetilde\nabla+\widetilde\phi$ is a
$\overline\partial$-connection on $\mathcal{A}^m\,$. The associated connection on $\mathcal{E}$ is
\begin{eqnarray*}
(id\otimes pr)\circ(\widetilde\nabla+\widetilde\phi)\circ i & = & \nabla+(id\otimes pr)\circ\phi\\
& = & \nabla^\prime
\end{eqnarray*}
i,e. $\nabla^\prime$ is induced by the $\overline{\partial}$-connection $\widetilde\nabla+\widetilde\phi$ on the free module $\mathcal{A}^m$.
\end{proof}

\begin{proposition}\label{holomorphic structure on free module}
Any free module over $\mathcal{A}$ is a holomorphic vector bundle.
\end{proposition}
\begin{proof}
Let $\mathcal{A}^m$ be a free module over $\mathcal{A}$ of rank $m$. Since $\Omega^{0,1}_{\partial,\overline{\partial}}(\mathcal{A})
\otimes_\mathcal{A}\mathcal{A}^m\cong\left(\Omega^{0,1}_{\partial,\overline{\partial}}(\mathcal{A})\right)^m$, define
\begin{align*}
\nabla_0:\mathcal{A}^m &\longrightarrow \Omega^{0,1}_{\partial,\overline{\partial}}(\mathcal{A})\otimes_\mathcal{A}\mathcal{A}^m\\
(a_1,\ldots,a_m) &\longmapsto \left([\,\overline{\partial},a_1],\ldots,[\,\overline{\partial},a_m]\right)
\end{align*}
It is easy to check that $\nabla_0$ is a $\overline{\partial}$-connection. Let $\{e_1,\ldots,e_m\}$ denotes the standard free $\mathcal{A}$-module
basis of $\mathcal{A}^m$. Then, the associated curvature becomes
\begin{eqnarray*}
\varTheta_{\nabla_0}(a_1,\ldots,a_m) & = & \nabla_0\left([\,\overline{\partial},a_1],\ldots,[\,\overline{\partial},a_m]\right)\\
& = & \sum_{j=1}^m\nabla_0([\,\overline{\partial},a_j]\otimes e_j)\\
& = & \sum_{j=1}^m-[\,\overline{\partial},a_j]\nabla_0(e_j)+[\,\overline{\partial},1][\,\overline{\partial},a_j]\otimes e_j\\
& = & 0
\end{eqnarray*}
since, $\nabla_0(e_j)=0$. Hence, $\nabla_0$ is flat $\overline{\partial}$-connection. This shows that $(\mathcal{A}^m,\nabla_0)$ is a holomorphic
vector bundle over $\mathcal{A}$.
\end{proof}

\begin{corollary}\label{space of holomorphic section}
The space of holomorphic sections of any free module $\mathcal{E}_0=\mathcal{A}^m$ over $\mathcal{A}$ is a free $\mathcal{O}(\mathcal{A})$-module
of rank $m$.
\end{corollary}
\begin{proof}
Let $\mathcal{E}_0=\mathcal{A}^m$ be a free module over $\mathcal{A}$ of rank $m$. Then $(\mathcal{E}_0\,,\nabla_0)$ is a holomorphic vector bundle over
$\mathcal{A}$ by Propn. (\ref{holomorphic structure on free module}). Now, the space of holomorphic sections becomes
\begin{eqnarray*}
H^0(\mathcal{E}_0\,,\nabla_0) & = & Ker\left\{\nabla_0:\mathcal{A}^m\longrightarrow\left(\Omega^{0,1}_{\partial,\overline{\partial}}
(\mathcal{A})\right)^m\right\}\\
& = & \{(a_1,\ldots,a_m):[\,\overline{\partial},a_j]=0\,\,\forall\,j=1,\ldots,m\}\\
& = & \{(a_1,\ldots,a_m):\,a_j\in\mathcal{O}(\mathcal{A})\,\forall\,j=1,\ldots,m\}\\
& \cong & \mathcal{O}(\mathcal{A})^m
\end{eqnarray*}
i,e. $H^0(\mathcal{E}_0\,,\nabla_0)$ is a free $\mathcal{O}(\mathcal{A})$-module of rank$=rank(\mathcal{E}_0)$.
\end{proof}
\medskip

\subsection{Space of differential forms on noncommutative \texorpdfstring{$2n$}{Lg}-tori}\texorpdfstring{$\newline$}{Lg}
Now, we concentrate on the particular case of noncommutative even dimensional torus. We will work with the holomorphic structure obtained in Th.
(\ref{N=(2,2) from dynamical system}). Since the algebra is concretely prescribed, one can explicitly compute the associated bimodule of noncommutative
space of $N=(1,1)$ differential forms (recall from Section $[2.2.2]$ in \cite{FGR2}) and complex differential forms (recall from Section $[2.3.2]$
in \cite{FGR2}).

\begin{definition}[\cite{Rfl1}]
Let $\mathscr{A}$ be the universal $C^*$-algebra generated by $2n$ many unitaries $\,U_1,\ldots,U_{2n}$ satisfying $\,U_jU_\ell=exp(2\pi i
\Theta_{\ell j})U_\ell U_j$, where $\Theta$ is a real $2n\times 2n$ skew-symmetric matrix such that  the lattice $\wedge_\Theta$ generated by its
columns makes $\wedge_\Theta+\mathbb{Z}^{2n}$ dense in $\mathbb{R}^{2n}$. The compact connected Lie group $\mathbb{T}^{2n}$ acts on $\mathscr{A}$ by
$\alpha_{\bf z}(U_\ell)=z_\ell U_\ell,\,\ell=1,\ldots,2n$. Let $\mathcal{A}_\Theta$ denotes the smooth subalgebra of $\mathscr{A}$ under this action.
Via Fourier transform one obtains
\begin{center}
$\mathcal{A}_\Theta:=\left\{\sum\alpha_{(j_1,\ldots,j_{2n})}U_1^{j_1}\ldots U_{2n}^{j_{2n}}\,:\,\alpha_{(j_1,\ldots,j_{2n})}\in\mathbb{S}
(\mathbb{Z}^{2n})\right\}\,.$
\end{center}
Then, $\mathcal{A}_\Theta$ is a unital spectrally invariant subalgebra of $\mathscr{A}$, called the noncommutative $2n$-torus.
\end{definition}

\begin{proposition}\label{N=(1,1) differential form for torus}
For the noncommutative $2n$-torus $\mathcal{A}_\Theta$, as an $\mathcal{A}_\Theta$-bimodule, we have
\begin{enumerate}
\item[(1)] $\Omega^0_d(\mathcal{A}_\Theta)\cong\mathcal{A}_\Theta$,
\item[(2)] $\Omega^\ell_d(\mathcal{A}_\Theta):=span\{a\prod_{j=1}^\ell[d,b_j]:a,b_j\in\mathcal{A}_\Theta\}\cong\mathcal{A}_\Theta^{\frac{2n!}
{\ell!(2n-\ell)!}}\,\,\,\forall\,1\leq\ell\leq 2n$,
\item[(3)] $\Omega^\ell_d(\mathcal{A}_\Theta)\cong\{0\}\,\,\forall\,\ell>2n$.
\end{enumerate}
\end{proposition}
\begin{proof}
Part $(1)$ is obvious. Recall that $d=\frac{1}{2}(\mathfrak{D}-i\overline{\mathfrak{D}})$, where $\mathfrak{D}=1\otimes D$ and $\overline{\mathfrak{D}}
=-\varepsilon^\prime D\otimes\sigma$ acting on $\mathcal{H}=\overline{\mathcal{E}\otimes_\mathcal{A}\mathcal{E}}\cong L^2(\mathcal{A},\tau)^{N^2}$
(see Propn. [\ref{an isomorphism of Hilbert spaces}] and [\ref{self-adjoint operator}]). Hence, for
\begin{center}
$d=\frac{1}{2}\sum_{j=1}^{2n}(\partial_j\otimes 1\otimes\gamma_j+i\varepsilon^\prime\partial_j\otimes\gamma_j\otimes\sigma)$
\end{center}
we see that
\begin{eqnarray*}
[d,a] & = & \frac{1}{2}\sum_{j=1}^{2n}(\partial_j(a)\otimes 1\otimes\gamma_j+i\varepsilon^\prime\partial_j(a)\otimes\gamma_j\otimes\sigma)\\
& = & \sum_{j=1}^{2n}\partial_j(a)\otimes\left(1\otimes \frac{1}{2}\gamma_j+\frac{i\varepsilon^\prime}{2}\gamma_j\otimes\sigma\right)\,.
\end{eqnarray*}
We claim that the set $\{1\otimes \frac{1}{2}\gamma_j+\frac{i\varepsilon^\prime}{2}\gamma_j\otimes\sigma:j=1,\ldots,2n\}$ is a linearly independent
subset of $M_N(\mathbb{C})\otimes M_N(\mathbb{C})$. Consider
\begin{eqnarray}\label{equation 1}
\sum_{j=1}^{2n}\alpha_j(1\otimes \frac{1}{2}\gamma_j+\frac{i\varepsilon^\prime}{2}\gamma_j\otimes\sigma)=0
\end{eqnarray}
with $\alpha_j\in\mathbb{C}$ for all $j$. Multiplying this equation (\ref{equation 1}) by $1\otimes\sigma$ from the right, and then $1\otimes
\sigma$ from the left we get
\begin{eqnarray}\label{equation 2}
\sum_{j=1}^{2n}\alpha_j(-1\otimes \frac{1}{2}\gamma_j+\frac{i\varepsilon^\prime}{2}\gamma_j\otimes\sigma)=0\,.
\end{eqnarray}
Now, $(\ref{equation 1})-(\ref{equation 2})$ gives us
\begin{center}
$\sum_{j=1}^{2n}1\otimes\alpha_j\gamma_j=0$
\end{center}
in $M_N(\mathbb{C})\otimes M_N(\mathbb{C})$. Since, $\{\gamma_1,\ldots,\gamma_{2n}\}$ is a linearly independent subset of $M_N(\mathbb{C})$ we get
$\alpha_j=0$ for all $j$ proving our claim. Hence, the following map
\begin{align*}
\Phi\,:\,\Omega^1_d(\mathcal{A}_\varTheta) &\longrightarrow \mathcal{A}_\Theta^{2n}\\
a[d,b] &\longmapsto (a\partial_1(b),\ldots,a\partial_{2n}(b))
\end{align*}
is an injective $\mathcal{A}_\Theta$-bimodule map. Now, for any $(0,\ldots,a,\ldots,0)\in\mathcal{A}_\Theta^{2n}$ with $a$ in the $j$-th place, the
element $aU_j^*\delta U_j\in\Omega^1(\mathcal{A}_\Theta),\,\delta$ being the universal differential, descends to $aU_j^*[d,U_j]\in\Omega^1_d
(\mathcal{A}_\Theta)$ and $\Phi(aU_j^*[d,U_j])=(0,\ldots,a,\ldots,0)$, proving surjectivity of $\Phi$. This concludes part $(2)$ for $\ell=1$. For
arbitrary $1\leq\ell\leq 2n$, first observe that
\begin{center}
$\left(1\otimes \frac{1}{2}\gamma_j+\frac{i\varepsilon^\prime}{2}\gamma_j\otimes\sigma\right)^2=0\quad,\quad\left\{(1\otimes \frac{1}{2}\gamma_j+
\frac{i\varepsilon^\prime}{2}\gamma_j\otimes\sigma)\,,\,(1\otimes \frac{1}{2}\gamma_r+\frac{i\varepsilon^\prime}{2}\gamma_r\otimes\sigma)\right\}=0$
\end{center}
for any $1\leq j\neq r\leq 2n$. Hence, for $a,b\in\mathcal{A}_\Theta$,
\begin{eqnarray*}
&  & [d,a][d,b]\\
& = & \sum_{1\leq j<r\leq 2n}^{2n}(\partial_j(a)\partial_r(b)-\partial_r(a)\partial_j(b))\otimes\left(1\otimes\frac{1}{2}\gamma_j+
\frac{i\varepsilon^\prime}{2}\gamma_j\otimes\sigma\right)\left(1\otimes\frac{1}{2}\gamma_r+\frac{i\varepsilon^\prime}{2}\gamma_r\otimes\sigma\right)\,.
\end{eqnarray*}
Same argument as in the case of $\ell=1$ will now show that $\Omega^2_d(\mathcal{A}_\Theta)\cong\mathcal{A}_\Theta^{\frac{2n!}{2(2n-2)!}}$.
By induction on $1\leq\ell\leq 2n$ one concludes Part $(2)$ and Part $(3)$ simultaneously.
\end{proof}

\begin{lemma}\label{complex differential form for torus}
For the noncommutative $2n$-torus $\mathcal{A}_\Theta$ with $n>1$, as an $\mathcal{A}_\Theta$-bimodule, we have
\begin{enumerate}
\item[(1)] $\Omega^{0,0}_{\partial,\overline{\partial}}(\mathcal{A}_\Theta)\cong\mathcal{A}_\Theta$,
\item[(2)] $\Omega^{1,0}_{\partial,\overline{\partial}}(\mathcal{A}_\Theta)\cong\mathcal{A}_\Theta^{n}$,
\item[(3)] $\Omega^{0,1}_{\partial,\overline{\partial}}(\mathcal{A}_\Theta)\cong\mathcal{A}_\Theta^{n}$,
\item[(4)] $\Omega^{2,0}_{\partial,\overline{\partial}}(\mathcal{A}_\Theta)\cong\mathcal{A}_\Theta^{\frac{n(n-1)}{2}}$,
\item[(5)] $\Omega^{0,2}_{\partial,\overline{\partial}}(\mathcal{A}_\Theta)\cong\mathcal{A}_\Theta^{\frac{n(n-1)}{2}}$,
\item[(6)] The product map $\Omega^{0,1}_{\partial,\overline{\partial}}(\mathcal{A}_\Theta)\times\Omega^{0,1}_{\partial,\overline{\partial}}
(\mathcal{A}_\Theta)\longrightarrow\Omega^{0,2}_{\partial,\overline{\partial}}(\mathcal{A}_\Theta)$ is given by
\begin{center}
$(a_1,\ldots,a_n).(b_1,\ldots,b_n)\longmapsto((a_pb_q-a_qb_p)_{1\leq p<q\leq n})\,.$
\end{center}
\end{enumerate}
\end{lemma}
\begin{proof}
Part $(1)$ is obvious. For part $(2)$, from Thm. (\ref{N=(2,2) from dynamical system}) we get that
\begin{eqnarray*}
[\partial,a] & = & \frac{1}{2}[d-id_2,a]\\
& = & \frac{1}{4}\left(\sum_{j=1}^{2n}\partial_j(a)\otimes 1\otimes\gamma_j-\sum_{\ell=1,\,\ell\,odd}^{2n}i\partial_\ell(a)\otimes 1\otimes
\gamma_{\ell+1}+\sum_{\ell=1,\,\ell\,odd}^{2n}i\partial_{\ell+1}(a)\otimes 1\otimes\gamma_\ell\right)\\
&  & +\frac{i\varepsilon^\prime}{4}\left(\sum_{j=1}^{2n}\partial_j(a)\otimes\gamma_j\otimes\sigma-\sum_{\ell=1,\,\ell\,odd}^{2n}i\partial_\ell(a)
\otimes\gamma_{\ell+1}\otimes\sigma+\sum_{\ell=1,\,\ell\,odd}^{2n}i\partial_{\ell+1}(a)\otimes\gamma_\ell\otimes\sigma\right)\\
& = & \frac{1}{4}\left(\sum_{\ell=1,\,\ell\,odd}^{2n}(\partial_{\ell+1}-i\partial_\ell)(a)\otimes 1\otimes\gamma_{\ell+1}+(\partial_\ell+i
\partial_{\ell+1})(a)\otimes 1\otimes\gamma_\ell\right)\\
&  & +\frac{i\varepsilon^\prime}{4}\left(\sum_{\ell=1,\,\ell\,odd}^{2n}(\partial_{\ell+1}-i\partial_\ell)(a)\otimes\gamma_{\ell+1}\otimes\sigma+
(\partial_\ell+i\partial_{\ell+1})(a)\otimes\gamma_\ell\otimes\sigma\right)\\
& = & \sum_{\ell=1,\,\ell\,odd}^{2n}\frac{1}{4}(\partial_{\ell+1}-i\partial_\ell)(a)\otimes 1\otimes(\gamma_{\ell+1}+i\gamma_\ell)+\frac{i
\varepsilon^\prime}{4}(\partial_{\ell+1}-i\partial_\ell)(a)\otimes(\gamma_{\ell+1}+i\gamma_\ell)\otimes\sigma
\end{eqnarray*}
for all $a\in\mathcal{A}_\Theta$. It can be verified (same way as in Propn. [\ref{N=(1,1) differential form for torus}]) that the set $\{1\otimes
\frac{1}{4}(\gamma_{\ell+1}+i\gamma_\ell)+\frac{i\varepsilon^\prime}{4}(\gamma_{\ell+1}+i\gamma_\ell)\otimes\sigma:\ell=1,\ldots,2n,\,\ell\,is\,odd\}$
is a linearly independent subset of $M_N(\mathbb{C})\otimes M_N(\mathbb{C})$. Hence, the following map
\begin{align*}
\Phi:\Omega^{1,0}_{\partial,\overline{\partial}}(\mathcal{A}_\Theta) &\longrightarrow \mathcal{A}_\Theta^{n}\\
a[\partial,b\,] &\longmapsto \sum_{\ell=1,\,\ell\,odd}^{2n}\left(0,\ldots,\,\underbrace{a\partial_{\ell+1}(b)-ia\partial_\ell(b)}_{\frac{\ell+1}
{2}-th\,\,place}\,,\ldots,0\right)
\end{align*}
is an injective $\mathcal{A}_\Theta$-bimodule map. For arbitrary $\xi=(0,\ldots,a,\ldots,0)\in\mathcal{A}_\Theta^{n}$ with $a$ in the $(\ell+1)/2$-th
place, $\Phi(aU_{\ell+1}^*[\partial,U_{\ell+1}])=\xi$. This shows that $\Phi$ is surjective, concluding Part $(2)$. Part $(3)$ follows similarly
since,
\begin{eqnarray*}
[\,\overline{\partial},a] & = & \frac{1}{2}[d+id_2,a]\\
& = & \frac{1}{4}\left(\sum_{j=1}^{2n}\partial_j(a)\otimes 1\otimes\gamma_j+\sum_{\ell=1,\,\ell\,odd}^{2n}i\partial_\ell(a)\otimes 1\otimes
\gamma_{\ell+1}-\sum_{\ell=1,\,\ell\,odd}^{2n}i\partial_{\ell+1}(a)\otimes 1\otimes\gamma_\ell\right)\\
&  & +\frac{i\varepsilon^\prime}{4}\left(\sum_{j=1}^{2n}\partial_j(a)\otimes\gamma_j\otimes\sigma+\sum_{\ell=1,\,\ell\,odd}^{2n}i\partial_\ell(a)
\otimes\gamma_{\ell+1}\otimes\sigma-\sum_{\ell=1,\,\ell\,odd}^{2n}i\partial_{\ell+1}(a)\otimes\gamma_\ell\otimes\sigma\right)\\
& = & \sum_{\ell=1,\,\ell\,odd}^{2n}\frac{1}{4}(\partial_{\ell+1}+i\partial_\ell)(a)\otimes 1\otimes(\gamma_{\ell+1}-i\gamma_\ell)+\frac{i
\varepsilon^\prime}{4}(\partial_{\ell+1}+i\partial_\ell)(a)\otimes(\gamma_{\ell+1}-i\gamma_\ell)\otimes\sigma
\end{eqnarray*}
for all $a\in\mathcal{A}_\Theta$.\\
For Part $(5)$, denote $\delta_j:=\partial_{2j}+i\partial_{2j-1}$ and $\eta_j:=\gamma_{2j}-i\gamma_{2j-1}$ for $j=1,\ldots,n$. Observe that
\begin{center}
$\eta_j^2=0\quad,\quad\{\eta_p,\eta_q\}=0\,\,\forall\,p\neq q\,.$
\end{center}
So by part $(3)$ we see that
\begin{center}
$[\,\overline{\partial},a]=\sum_{j=1}^n\frac{1}{4}\delta_j(a)\otimes 1\otimes\eta_j+\frac{i\varepsilon^\prime}{4}\delta_j(a)\otimes\eta_j\otimes\sigma\,.$
\end{center}
Hence, for arbitrary $a,b\in\mathcal{A}_\Theta$,
\begin{eqnarray*}
&  & [\,\overline{\partial},a][\,\overline{\partial},b]\\
& = & \frac{1}{16}\sum_{\ell<r}(\delta_r(a)\delta_\ell(b)-\delta_\ell(a)\delta_r(b))\otimes
(1\otimes\eta_\ell\eta_r-\eta_\ell\eta_r\otimes 1)\\
&  & \,\,+\frac{i\varepsilon^\prime}{16}\sum_{\ell\neq r}(\delta_r(a)\delta_\ell(b))\otimes(\eta_\ell\otimes\eta_r-\eta_r\otimes\eta_\ell)(1\otimes\sigma)\\
& = & \sum_{\ell<r}(\delta_r(a)\delta_\ell(b)-\delta_\ell(a)\delta_r(b))\otimes\left(\frac{1}{16}(1\otimes\eta_\ell\eta_r-\eta_\ell\eta_r
\otimes 1)+\frac{i\varepsilon^\prime}{16}(\eta_\ell\otimes\eta_r\sigma-\eta_r\otimes\eta_\ell\sigma)\right)
\end{eqnarray*}
The set $\{\frac{1}{16}(1\otimes\eta_\ell\eta_r-\eta_\ell\eta_r\otimes 1)+\frac{i\varepsilon^\prime}{16}(\eta_\ell\otimes\eta_r\sigma-\eta_r\otimes
\eta_\ell\sigma):1\leq\ell<r\leq n\}$ can be easily seen to be a linearly independent subset of $M_N(\mathbb{C})\otimes M_N(\mathbb{C})$. Hence, the
following map
\begin{align*}
\Phi:\Omega^{0,2}_{\partial,\overline{\partial}}(\mathcal{A}_\Theta) &\longrightarrow \mathcal{A}_\Theta^{\frac{n(n-1)}{2}}\\
a[\,\overline\partial,b\,][\,\overline\partial,c\,] &\longmapsto \left((a\delta_r(b)\delta_\ell(c)-a\delta_\ell(b)\delta_r(c))_{1\leq\ell<r\leq n}\right)
\end{align*}
is an injective $\mathcal{A}_\Theta$-bimodule map. To see surjectivity, observe that for any $a\in\mathcal{A}_\Theta$ in $(\ell,r)$-position with $\ell<r$,
\begin{center}
$\Phi:aU_{2\ell}^*U_{2r}^*[\,\overline\partial,U_{2r}][\,\overline\partial,U_{2\ell}]\longmapsto a\,.$
\end{center}
This completes Part $(5)$, and Part $(4)$ follows similarly.\\
For Part $(6)$, starting with $(a_1,\ldots,a_n),\,(b_1,\ldots,b_n)\in\mathcal{A}_\Theta^n$ first obtain their respective inverse image in
$\Omega^{0,1}_{\partial,\overline{\partial}}(\mathcal{A}_\Theta)$ using Part $(3)$, then take the product to get an element in $\Omega^{0,2}_{\partial,
\overline{\partial}}(\mathcal{A}_\Theta)$ and finally use the isomorphism in Part $(5)$ to find its image in $\mathcal{A}_\Theta^{\frac{n(n-1)}{2}}$.
We left this for the reader to verify.
\end{proof}

\begin{corollary}\label{product rule of basis}
If $\{e_1,\ldots,e_n\}$ denotes the standard free module basis of $\Omega^{0,1}_{\partial,\overline{\partial}}(\mathcal{A}_\Theta)\cong
\mathcal{A}_\Theta^n$ then $\{e_\ell e_r:1\leq\ell<r\leq n\}$ is a free module basis of $\Omega^{0,2}_{\partial,\overline{\partial}}
(\mathcal{A}_\Theta)\cong\mathcal{A}_\Theta^{\frac{n(n-1)}{2}}$. Moreover, $e_\ell e_r+e_re_\ell=e_\ell^2=0$ for all $1\leq\ell<r\leq n$.
\end{corollary}
\begin{proof}
Follows from Part $(3),\,(4)$ and $(5)$ in the previous Lemma (\ref{complex differential form for torus}).
\end{proof}

\begin{proposition}\label{higher complex differential form for torus}
For the noncommutative $2n$-torus $\mathcal{A}_\Theta$, as an $\mathcal{A}_\Theta$-bimodule, we have
\begin{enumerate}
\item $\Omega^{\ell,0}_{\partial,\overline{\partial}}(\mathcal{A}_\Theta)\cong\mathcal{A}_\Theta^{\frac{n!}{\ell!(n-\ell)!}}\,\,\forall\,1\leq\ell\leq n\,$,
\item $\Omega^{0,\ell}_{\partial,\overline{\partial}}(\mathcal{A}_\Theta)\cong\mathcal{A}_\Theta^{\frac{n!}{\ell!(n-\ell)!}}\,\,\forall\,1\leq\ell\leq n\,$,
\item $\Omega^{\ell,0}_{\partial,\overline{\partial}}(\mathcal{A}_\Theta)=\Omega^{0,\ell}_{\partial,\overline{\partial}}(\mathcal{A}_\Theta)=\{0\}\,\,\forall\,\ell>n\,$,
\item $\Omega_d^r(\mathcal{A}_\Theta)\cong\bigoplus_{p+q=r}\Omega^{p,q}_{\partial,\overline{\partial}}(\mathcal{A}_\Theta)\,$.
\end{enumerate}
\end{proposition}
\begin{proof}
The case of $n=1$ should be treated separately. In this case of $\mathcal{A}_\theta$,
\begin{center}
$[\,\overline{\partial},a]=\frac{1}{4}(\partial_2+i\partial_1)(a)\otimes 1\otimes(\gamma_2-i\gamma_1)+\frac{i\varepsilon^\prime}{4}(\partial_2+
i\partial_1)(a)\otimes(\gamma_2-i\gamma_1)\otimes\sigma$
\end{center}
for all $a\in\mathcal{A}_\theta$. Since, $(\gamma_2-i\gamma_1)^2=\{\gamma_2-i\gamma_1,\sigma\}=0$, one gets that $[\,\overline{\partial},a]
[\,\overline{\partial},b]=0$ for all $a,b\in\mathcal{A}_\theta$. Part $(1,2,3)$ now follows by induction on $\ell$ in Lemma (\ref{complex differential form for torus}),
similarly as in Propn. (\ref{N=(1,1) differential form for torus}). To show Part $(4)$ recall from Propn. $[2.33]$ in (\cite{FGR2}) that it is
enough to show $[T,\omega]\in\Omega_d^1(\mathcal{A}_\Theta)$ for all $\omega\in\Omega_d^1(\mathcal{A}_\Theta)$, where $T=\frac{1}{2}(\mathcal{T}-
i\mathcal{I})$ is as in Propn. (\ref{sufficient condition for complex structure}). Observe that if $\omega=a[d,b]$ then $[T,\omega]=a[\partial,b]
\in\Omega^{1,0}_{\partial,\overline{\partial}}(\mathcal{A}_\Theta)$. By Propn. (\ref{N=(1,1) differential form for torus}) and Lemma
(\ref{complex differential form for torus}) we see that
\begin{center}
$\Omega_d^1(\mathcal{A}_\Theta)\cong\Omega^{1,0}_{\partial,\overline{\partial}}(\mathcal{A}_\Theta)\bigoplus\Omega^{0,1}_{\partial,\overline{\partial}}
(\mathcal{A}_\Theta)$
\end{center}
as $\mathcal{A}_\Theta$-bimodules. Hence, we conclude that $[T,\omega]\in\Omega_d^1(\mathcal{A}_\Theta)$ for all $\omega\in\Omega_d^1(\mathcal{A}_\Theta)$.
This concludes Part $(4)$.
\end{proof}
\medskip

\subsection{Holomorphic vector bundle over noncommutative \texorpdfstring{$2n$}{Lg}-tori}\texorpdfstring{$\newline$}{Lg}
Results of this subsection are partly motivated by (\cite{BS}) and computations in (\cite{CG1}).

\begin{proposition}\label{holomorphic elements}
The algebra $\mathcal{O}(\mathcal{A}_\Theta)$ of holomorphic elements in $\mathcal{A}_\Theta$ is $\,\mathbb{C}$.
\end{proposition}
\begin{proof}
From Lemma (\ref{complex differential form for torus}),
\begin{align*}
\overline{\partial}:\mathcal{A}_\Theta &\longrightarrow \Omega^{0,1}_{\partial,\overline{\partial}}(\mathcal{A}_\Theta)\cong\mathcal{A}_\Theta^n\\
a &\longmapsto ((\partial_2+i\partial_1)(a),\ldots,(\partial_{2n}+i\partial_{2n-1})(a))\,.
\end{align*}
Hence, by Defn. (\ref{algebra of holomorphic elements}),
\begin{center}
$\mathcal{O}(\mathcal{A}_\Theta)=\{\,a\in\mathcal{A}_\Theta:(\partial_{j+1}+i\partial_j)(a)=0\,\,\forall\,j=1,\ldots,2n\,;\,j\,\,odd\}\,.$
\end{center}
Arbitrary $\,a\in\mathcal{A}_\Theta$ is of the form $\sum_{(\ell_1,\ldots,\ell_{2n})\in\mathbb{Z}^{2n}}\,\alpha_{\ell_1,\ldots,\ell_{2n}}U_1^{\ell_1}
\ldots U_{2n}^{\ell_{2n}}$ where $\alpha_{\ell_1,\ldots,\ell_{2n}}\in\mathbb{S}(\mathbb{Z}^{2n})$ and hence, for any $j\in\{1,\ldots,2n\}$ with $j$ odd we have
\begin{eqnarray*}
(\partial_{j+1}+i\partial_j)(a) & = & \sum\,(\ell_{j+1}+i\ell_j)\alpha_{\ell_1,\ldots,\ell_{2n}}U_1^{\ell_1}\ldots U_{2n}^{\ell_{2n}}\,.
\end{eqnarray*}
This expression is equal to zero implies that $(\ell_{j+1}+i\ell_j)\alpha_{\ell_1,\ldots,\ell_{2n}}=0$. Hence, $\ell_{j+1}=\ell_j=0$. Thus, $a$ is
of the form $\sum\,\alpha_{\ell_1,\ldots,\ell_{2n}}U_1^{\ell_1}\ldots\widehat{U_j^{\ell_j}}\widehat{U_{j+1}^{\ell_{j+1}}}\ldots U_{2n}^{\ell_{2n}}$.
This is true for all $j\in\{1,\ldots,2n\},\,j$ is odd. Hence, we conclude that $a\in\mathbb{C}1$, which proves $\mathcal{O}(\mathcal{A}_\Theta)\cong\mathbb{C}$.
\end{proof}

\begin{corollary}
Space of holomorphic sections of any free module $\mathcal{E}_0=\mathcal{A}_\Theta^m$ over $\mathcal{A}_\Theta$ is $\mathbb{C}^m$.
\end{corollary}
\begin{proof}
Follows from Cor. (\ref{space of holomorphic section}) and previous  Propn. (\ref{holomorphic elements}).
\end{proof}

\begin{lemma}
$\Omega_d^r(\mathcal{A}_\Theta),\,\Omega^{\ell,0}_{\partial,\overline{\partial}}(\mathcal{A}_\Theta),\,\Omega^{0,\ell}_{\partial,\overline{\partial}}
(\mathcal{A}_\Theta)$ all are holomorphic vector bundles over $\mathcal{A}_\Theta$.
\end{lemma}
\begin{proof}
Follows from Propn. (\ref{holomorphic structure on free module}) and (\ref{N=(1,1) differential form for torus}\,,\,\ref{higher complex differential form for torus}).
\end{proof}

\begin{proposition}\label{criterion for obtaining holomorphic structure}
A holomorphic structure on a f.g.p. left module $\mathcal{E}$ over $\mathcal{A}_\Theta$ is given by $n$-tuple $(\nabla_1,\ldots,\nabla_n)$ of
$\mathbb{C}$-linear maps $\nabla_j:\mathcal{E}\longrightarrow\mathcal{E}$ such that the following conditions are satisfied
\begin{enumerate}
\item $\nabla_j(a\xi)=a\nabla_j(\xi)+\delta_j(a)\xi\quad\forall\,a\in\mathcal{A}_\Theta$,
\item $[\nabla_\ell,\nabla_r]=0\quad\forall\,1\leq l<r\leq n$.
\end{enumerate}
where, $\delta_j=\partial_{2j}+i\partial_{2j-1}$.
\end{proposition}
\begin{proof}
Recall from Lemma (\ref{complex differential form for torus}) that $\Omega_{\partial,\overline{\partial}}^{0,1}(\mathcal{A}_\Theta)\cong\mathcal{A}_\Theta^n$.
Hence, any $\overline{\partial}$-connection $\nabla:\mathcal{E}\longrightarrow\Omega_{\partial,\overline{\partial}}^{0,1}(\mathcal{A}_\Theta)\otimes
\mathcal{E}$ on $\mathcal{E}$ will be implemented by $n$-tuple of $\mathbb{C}$-linear maps $(\nabla_1,\ldots,\nabla_n)$ with each $\nabla_j:
\mathcal{E}\longrightarrow\mathcal{E}$. Since, $\nabla$ is a $\overline{\partial}$-connection, it is easy to verify that
\begin{center}
$\nabla_j(a\xi)=a\nabla_j(\xi)+\delta_j(a)\xi\quad\forall\,a\in\mathcal{A}_\Theta$ and $\xi\in\mathcal{E}$
\end{center}
where, $\delta_j=\partial_{2j}+i\partial_{2j-1}$. If $\nabla$ induces a holomorphic structure on $\mathcal{E}$ then $\varTheta_\nabla=0$. Now, if
$\{e_1,\ldots,e_n\}$ denotes the standard free module basis of $\,\Omega_{\partial,\overline{\partial}}^{0,1}(\mathcal{A}_\Theta)\cong\mathcal{A}_\Theta^n$
then observe from Lemma (\ref{complex differential form for torus}) that $\overline{\partial}^{\,\prime}:\Omega_{\partial,\overline{\partial}}^{0,1}
(\mathcal{A}_\Theta)\longrightarrow\Omega_{\partial,\overline{\partial}}^{0,2}(\mathcal{A}_\Theta)$ satisfies $\overline{\partial}^{\,\prime}(e_j)=0$.
Hence, for any $\xi\in\mathcal{E}$, we have
\begin{eqnarray*}
\varTheta_\nabla(\xi) & = & \sum_{j=1}^n\nabla(e_j\otimes\nabla_j(\xi))\\
& = & \sum_{j=1}^n-e_j\nabla(\nabla_j(\xi))+\overline{\partial}^{\,\prime}(e_j)\otimes\nabla_j(\xi)\\
& = & \sum_{\ell,j=1}^n-e_je_\ell\otimes\nabla_\ell(\nabla_j(\xi))\\
& = & \sum_{\ell<j}e_\ell e_j\otimes[\nabla_\ell,\nabla_j](\xi)
\end{eqnarray*}
because $e_pe_q+e_qe_p=0$ for $p\neq q$ and $e_p^2=0$ (Cor. [\ref{product rule of basis}]). Since, $\{e_\ell e_j:1\leq\ell<j\leq n\}$ is the standard
free module basis of $\Omega_{\partial,\overline{\partial}}^{0,2}(\mathcal{A}_\Theta)$ we get $\varTheta_\nabla=0$ if and only if $[\nabla_\ell,
\nabla_r]=0$ for all $1\leq\ell<r\leq n$. This fulfills our claim.
\end{proof}

Observe from Propn. (\ref{N=(1,1) differential form for torus}) and Lemma (\ref{complex differential form for torus}) that $\Omega_d^1
(\mathcal{A}_\Theta)\cong\Omega_{\partial,\overline\partial}^{1,0}(\mathcal{A}_\Theta)\bigoplus\Omega_{\partial,\overline\partial}^{0,1}
(\mathcal{A}_\Theta)$. This is in fact a orthogonal direct sum by Propn. $[2.35]$ in (\cite{FGR2}). Hence, any $\mathbb{C}$-linear map $\nabla:
\mathcal{E}\longrightarrow\Omega_d^1(\mathcal{A}_\Theta)\otimes_{\mathcal{A}_\Theta}\mathcal{E}$ satisfying $\nabla(a\xi)=a\nabla(\xi)+[d,a]\otimes
\xi$, i,e. a $d$-connection, can be written as $\nabla^{1,0}+\nabla^{0,1}$. Let $\pi^{1,0}$ and $\pi^{0,1}$ be the orthogonal projections onto
$\Omega_{\partial,\overline\partial}^{1,0}$ and $\Omega_{\partial,\overline\partial}^{0,1}$ respectively. Note that these are $\mathcal{A}_\Theta$-module maps.

\begin{proposition}
Let $\mathcal{E}$ be a f.g.p. left module over $\mathcal{A}_\Theta$ and $\nabla:\mathcal{E}\longrightarrow\Omega_d^1(\mathcal{A}_\Theta)
\otimes_{\mathcal{A}_\Theta}\mathcal{E}$ be a $d$-connection whose curvature has vanishing $(0,2)$-component. Then $\nabla$ induces a holomorphic
structure on $\mathcal{E}$. In particular, any flat $d$-connection induces a holomorphic structure on $\mathcal{E}$.
\end{proposition}
\begin{proof}
Let $\nabla$ be a $d$-connection and define
\begin{align*}
\nabla^\prime:\mathcal{E} &\longrightarrow \Omega_{\partial,\overline\partial}^{0,1}(\mathcal{A}_\Theta)\otimes_{\mathcal{A}_\Theta}\mathcal{E}\\
\xi &\longmapsto (\pi^{0,1}\otimes id)\nabla(\xi)\,.
\end{align*}
Since $\pi^{0,1}$ is a left $\mathcal{A}_\Theta$-module homomorphism, it is easy to observe that $\nabla^\prime$ is a $\overline\partial$-connection.
Observe that the associated curvatures satisfy the following relation
\begin{center}
$\varTheta_{\nabla^\prime}=(\pi^{0,2}\otimes id)\varTheta_\nabla\,$.
\end{center}
Hence, if the $(0,2)$-component of the curvature $\varTheta_\nabla$ vanishes then $\nabla^\prime$ is a flat $\overline{\partial}$-connection. For
detail verification follow the proof of Propn. $[4.7]$ in (\cite{BS}). In particular, if $\varTheta_\nabla$ itself is zero i,e. $\nabla$ is $d$-flat
then $\nabla^\prime$ induces a holomorphic structure on $\mathcal{E}$.
\end{proof}

If $(\mathcal{E},\nabla)$ is a holomorphic vector bundle over $\mathcal{A}$ then
\begin{center}
$0\longrightarrow\mathcal{E}\xrightarrow{\,\,\nabla\,\,}\Omega_{\partial,\overline\partial}^{0,1}(\mathcal{A})\otimes_\mathcal{A}\mathcal{E}
\xrightarrow{\,\,\nabla\,\,}\Omega_{\partial,\overline\partial}^{0,2}(\mathcal{A})\otimes_\mathcal{A}\mathcal{E}\xrightarrow{\,\,\nabla\,\,}\ldots\ldots$
\end{center}
is a cochain complex. The cohomology groups of this complex are denoted by $H^\bullet(\mathcal{E},\nabla)$. Recall from Defn. (\ref{holomorphic section})
that the zero-th cohomology is the space of holomorphic sections on $\mathcal{E}$. It follows from the definition of connection that each $H^\bullet
(\mathcal{E},\nabla)$ is a left $\mathcal{O}(\mathcal{A})$-module. Hence, for the case of noncommutative torus $\mathcal{A}=\mathcal{A}_\Theta$,
they are $\mathbb{C}$-vector spaces (by Propn. [\ref{holomorphic elements}]).

\begin{proposition}
Every short exact sequence
\begin{center}
$0\longrightarrow(\mathcal{E},\nabla_\mathcal{E})\xrightarrow{\,\,\phi\,\,}(\mathcal{F},\nabla_\mathcal{F})\xrightarrow{\,\,\psi\,\,}
(\mathcal{G},\nabla_\mathcal{G})\longrightarrow 0$
\end{center}
of holomorphic vector bundles over $\mathcal{A}_\Theta$ induces a long exact sequence
\begin{center}
$0\longrightarrow H^0(\mathcal{E},\nabla_\mathcal{E})\xrightarrow{\,\,\,\phi^*\,\,}H^0(\mathcal{F},\nabla_\mathcal{F})\xrightarrow{\,\,\,\psi^*\,\,}
H^0(\mathcal{G},\nabla_\mathcal{G})\xrightarrow{\,\,\,\overline\delta\,\,\,}H^1(\mathcal{E},\nabla_\mathcal{E})\xrightarrow{\,\,\,\phi^*\,\,}\ldots\ldots$
\end{center}
in cohomology of $\,\mathbb{C}$-vector spaces.
\end{proposition}
\begin{proof}
Since, $\Omega_{\partial,\overline\partial}^{0,\bullet}(\mathcal{A}_\Theta)$ are free modules over $\mathcal{A}_\Theta$ (Propn.
[\ref{higher complex differential form for torus}]) we get
\begin{center}
$0\longrightarrow\Omega_{\partial,\overline\partial}^{0,\bullet}(\mathcal{A}_\Theta)\otimes_{\mathcal{A}_\Theta}\mathcal{E}\xrightarrow{\,\,id
\otimes\phi\,\,}\Omega_{\partial,\overline\partial}^{0,\bullet}(\mathcal{A}_\Theta)\otimes_{\mathcal{A}_\Theta}\mathcal{F}\xrightarrow{\,\,id\otimes
\psi\,\,}\Omega_{\partial,\overline\partial}^{0,\bullet}(\mathcal{A}_\Theta)\otimes_{\mathcal{A}_\Theta}\mathcal{G}\longrightarrow 0$
\end{center}
is an exact sequence of cochain complexes which induces a long exact sequence in cohomology (See Propn. $[4.6]$ in \cite{BS}).
\end{proof}

\begin{theorem}
The category $\mathcal{H}o\ell(\mathcal{A}_\Theta)$ of holomorphic vector bundle over $\mathcal{A}_\Theta$ is an abelian category.
\end{theorem}
\begin{proof}
We follow the proof of Propn. $[4.5]$ in (\cite{BS}). Let $\,\Phi:(\mathcal{E}_1,\nabla_1)\rightarrow(\mathcal{E}_2,\nabla_2)$ be a morphism in
$\mathcal{H}o\ell(\mathcal{A}_\Theta)$. Consider the following exact sequence
\begin{center}
$0\longrightarrow Ker(\phi)\longrightarrow\mathcal{E}_1\xrightarrow{\,\,\phi\,\,}\mathcal{E}_2\longrightarrow Coker(\phi)\longrightarrow 0$
\end{center}
Denote $Ker(\phi)=K$ and $Coker(\phi)=C$. Since, $\,\Omega_{\partial,\overline{\partial}}^{0,1}(\mathcal{A}_\Theta)$ is a free module
(Lemma [\ref{complex differential form for torus}]), we get that the following sequence
\begin{center}
$0\longrightarrow\Omega_{\partial,\overline{\partial}}^{0,1}(\mathcal{A}_\Theta)\otimes_{\mathcal{A}_\Theta}K\longrightarrow
\Omega_{\partial,\overline{\partial}}^{0,1}(\mathcal{A}_\Theta)\otimes_{\mathcal{A}_\Theta}\mathcal{E}_1\xrightarrow{id\otimes\phi}
\Omega_{\partial,\overline{\partial}}^{0,1}(\mathcal{A}_\Theta)\otimes_{\mathcal{A}_\Theta}\mathcal{E}_2\longrightarrow
\Omega_{\partial,\overline{\partial}}^{0,1}(\mathcal{A}_\Theta)\otimes_{\mathcal{A}_\Theta}C\longrightarrow 0$
\end{center}
is exact. There are unique maps $\nabla_K:K\longrightarrow\Omega_{\partial,\overline{\partial}}^{0,1}(\mathcal{A}_\Theta)\otimes_{\mathcal{A}_\Theta}K$
obtained by restricting $\nabla_1$ to $K$ and $\nabla_C:C\longrightarrow\Omega_{\partial,\overline{\partial}}^{0,1}(\mathcal{A}_\Theta)
\otimes_{\mathcal{A}_\Theta}C$ induced by $\nabla_2$, making the following diagram
\begin{center}
\begin{tikzcd}
0\arrow{r} & K\arrow{r}\arrow{d}{\nabla_K} & \mathcal{E}_1\arrow{r}{\phi}\arrow{d}{\nabla_1} & \mathcal{E}_2\arrow{r}\arrow{d}{\nabla_2} &
C\arrow{r}\arrow{d}{\nabla_C} & 0 \\
0\arrow{r} & \Omega_{\partial,\overline{\partial}}^{0,1}\otimes_{\mathcal{A}_\Theta} K\arrow{r} & 
\Omega_{\partial,\overline{\partial}}^{0,1}\otimes_{\mathcal{A}_\Theta}\mathcal{E}_1\arrow{r}{1\otimes\phi} & 
\Omega_{\partial,\overline{\partial}}^{0,1}\otimes_{\mathcal{A}_\Theta}\mathcal{E}_2\arrow{r} & 
\Omega_{\partial,\overline{\partial}}^{0,1}\otimes_{\mathcal{A}_\Theta} C\arrow{r} & 0
\end{tikzcd}
\end{center}
commutative. Now, it is easy to check that $(K,\nabla_K)$ and $(C,\nabla_C)$ are respectively kernel and cokernel for $\Phi$ in $\mathcal{H}o\ell
(\mathcal{A}_\Theta)$. This concludes the proof.
\end{proof}
\medskip

\subsection{The case of noncommutative two-torus revisited}\texorpdfstring{$\newline$}{Lg}
In this subsection we revisit the case of noncommutative two-torus $\mathcal{A}_\theta$ studied earlier in (\cite{PS},\cite{P}), and obtain their
set-up as a special case of our general framework.

Recall from Part $(3)$ in Propn. (\ref{higher complex differential form for torus}) that the noncommutative space of complex two forms
\begin{center}
$\Omega_{\partial,\overline\partial}^{0,2}(\mathcal{A}_\theta):=span\{a[\,\overline{\partial},b][\,\overline{\partial},c]:a,b,c\in\mathcal{A}_\theta\}$
\end{center}
vanishes identically for the case of noncommutative two-torus. Because of this reason for any $\overline\partial$-connection $\nabla:\mathcal{E}
\longrightarrow\Omega_{\partial,\overline\partial}^{0,1}(\mathcal{A}_\theta)\otimes_{\mathcal{A}_\theta}\mathcal{E}$, the associated
$\overline\partial$-curvature $\varTheta_\nabla:\mathcal{E}\longrightarrow\Omega_{\partial,\overline\partial}^{0,2}(\mathcal{A}_\theta)
\otimes_{\mathcal{A}_\theta}\mathcal{E}$ is always zero i,e. $\nabla$ is always $\overline\partial$-flat. Also, as observed in Lemma
(\ref{complex differential form for torus}), we have
\begin{align*}
\Phi:\Omega_{\partial,\overline\partial}^{0,1}(\mathcal{A}_\theta) &\longrightarrow \mathcal{A}_\theta\\
a[\,\overline\partial,b] &\longmapsto a(\partial_2+i\partial_1)(b)
\end{align*}
is an $\mathcal{A}_\theta$-bimodule isomorphism. Hence, for any f.g.p. left $\mathcal{A}_\theta$-module $\mathcal{E}$ we get $\Omega_{\partial,
\overline\partial}^{0,1}(\mathcal{A}_\theta)\otimes_{\mathcal{A}_\theta}\mathcal{E}$ is canonically isomorphic with $\mathcal{E}$, since
$\mathcal{A}_\theta$ is unital. Therefore, a holomorphic structure on $\mathcal{E}$ is given by a $\mathbb{C}$-linear map $\nabla:\mathcal{E}
\longrightarrow\mathcal{E}$ such that
\begin{center}
$\nabla(a\xi)=a\nabla(\xi)+(\partial_2+i\partial_1)(a)\xi$
\end{center}
for all $\xi\in\mathcal{E}$ and $a\in\mathcal{A}_\theta$. For arbitrary $a\in
\mathcal{A}_\theta$ of the form $\sum_{(r_1,r_2)\in\mathbb{Z}^2}\alpha_{r_1,r_2}U_1^{r_1}U_2^{r_2}$ we see that
\begin{center}
$(\partial_2+i\partial_1)(a)=(r_2+ir_1)a\,.$
\end{center}
If we denote $\tau$ to be the purely imaginary number $i$ then the derivation on $\mathcal{A}_\theta$ defined by
\begin{center}
$\partial_\tau\left(\sum_{(r_1,r_2)\in\mathbb{Z}^2}\alpha_{r_1,r_2}U_1^{r_1}U_2^{r_2}\right):=2\pi i\sum_{(r_1,r_2)\in\mathbb{Z}^2}(r_1\tau+r_2)
\alpha_{r_1,r_2}U_1^{r_1}U_2^{r_2}$
\end{center}
is equal to $\,\Phi\circ[\,\overline\partial,.]$. This is the complex structure considered in (\cite{PS},\cite{P}) for $\mathcal{A}_\theta$, and we
see that the definition of holomorphic vector bundle given in (\cite{PS}) for the case of noncommutative two-torus $\mathcal{A}_\theta$ is a special
case of the general definition given in (Defn. [\ref{holomorphic vector bundle}]).

Moreover, the complex $\left(\Omega_{\partial,\overline\partial}^{0,\bullet}(\mathcal{A}_\theta)\otimes_{\mathcal{A}_\theta}\mathcal{E},\nabla\right)$
becomes just
\begin{center}
$0\longrightarrow\mathcal{E}\xrightarrow{\,\,\nabla\,\,}\mathcal{E}\longrightarrow 0$
\end{center}
and hence, the cohomology becomes
\begin{center}
$H^0(\mathcal{E},\nabla)=Ker\{\nabla:\mathcal{E}\longrightarrow\mathcal{E}\}\quad$ and $\quad H^1(\mathcal{E},\nabla)=
Coker\{\nabla:\mathcal{E}\longrightarrow\mathcal{E}\}$.
\end{center}
We see that this is the definition of the cohomology given in (\cite{PS}).
\bigskip

\section*{Open questions}
\begin{enumerate}
\item At present, it is not clear to us whether more K\"ahler structures exist in Th. (\ref{final theorem}). If we carefully look at the proof of
Th. (\ref{N=(2,2) from dynamical system}) and Propn. (\ref{total number of Kahler structure}) then we see that there may be several other choice
for the anti-selfadjoint operator $\mathcal{I}$ acting on $L^2(\mathcal{A},\tau)^{N^2}$, apart from those already produced here. Classifying all
possible K\"ahler structures should be an interesting investigation.
\item Under which condition(s) the $N=(2,2)$ K\"ahler spectral data obtained in Thm. (\ref{N=(2,2) from dynamical system}) extends further to
$N=(4,4)$ hyper-K\"ahler spectral data (Defn. $[2.37]$ in \cite{FGR2})? Certainly, one necessary condition would be $dim(G)=4k$, where $G=\mathbb{T}^{m}
\times\mathbb{R}^n$. But question is whether this is also the sufficient condition. Note that in the classical case, the $4k$-dimensional tori are
actually hyper-K\"ahler manifolds. We expect the same for the $4k$-dimensional noncommutative tori also.
\end{enumerate}
\bigskip

\section*{Acknowledgement}
Author gratefully acknowledges financial support of DST India through INSPIRE Faculty award (Award No. DST/INSPIRE/04/2015/000901).

\end{document}